\documentclass[11pt,twoside]{article}

\def\titlename{\huge Weighted discrete hypergroups}

\title{\titlename}

\def\authname{Mahmood Alaghmandan and Ebrahim Samei}
\author{{\normalsize\sc \authname}}

\usepackage{amsmath}
\usepackage{yfonts}

\usepackage{xcolor}
\definecolor{blue1}{RGB}{32,78,170}
\definecolor{blue2}{RGB}{93,92,160}
\definecolor{blue3}{RGB}{40,51,202}
\definecolor{blue4}{RGB}{0,0,0}
\definecolor{purple1}{RGB}{128,0,128}

\usepackage[all]{xy}

\definecolor{El}{rgb}{.4,.9,1}
\usepackage[pdftex,
            pdfauthor={Mamhood},
            linkcolor=blue1,
            citecolor=blue3,
            colorlinks=true]{hyperref}

\usepackage{titlesec}

\titleformat{\section}
  {\normalfont\fontsize{12}{15}\bfseries}{\thesection}{1em}{}

\usepackage[english]{babel}
\usepackage{blindtext}
\usepackage[scaled=.90]{helvet}
\usepackage{xcolor}
\usepackage{titlesec}
\titleformat{\chapter}[display]
  {\normalfont\sffamily\huge\bfseries\color{blue4}}
  {\chaptertitlename\ \thechapter}{20pt}{\Huge}
\titleformat{\section}
  {\normalfont\sffamily\Large\color{blue4}}
{\thesection}{1em}{}

 \titleformat{\subsection}
  {\normalfont\fontsize{12}{14}\sffamily\color{blue4}}
  {\thesubsection}{1em}{}

\newcommand{\ma}[1]{\textcolor{blue2}{\emph{\textsf{#1}}}}

             \newcounter{pulse}[section]
\numberwithin{pulse}{section}
\numberwithin{equation}{section}
\newtheorem{theorem}[pulse]{\bf \textsf{Theorem}}
\newtheorem{proposition}[pulse]{\bf \textsf{Proposition}}
\newtheorem{lemma}[pulse]{\bf \textsf{Lemma}}
\newtheorem{corollary}[pulse]{\bf \textsf{Corollary}}

\newtheorem{cor}[pulse]{\bf \textsf{Corollary}}

\newtheorem{dummy-eg}[pulse]{\bf \textsf{Example}}
\newtheorem{dummy-rem}[pulse]{\bf \textsf{Remark}}
\newenvironment{eg}{\begin{dummy-eg}\upshape}{\end{dummy-eg}\ignorespacesafterend}
\newenvironment{rem}{\begin{dummy-rem}\upshape}{\end{dummy-rem}\ignorespacesafterend}

\newtheorem{dummy-def}[pulse]{\bf \textsf{Definition}}

\newenvironment{proof}{\noindent{\it Proof.}\/}{\hfill$\Box$\newline\ignorespacesafterend}
\usepackage{amsmath,amssymb}
\newenvironment{dfn}{\begin{dummy-def}\upshape}{\end{dummy-def}\ignorespacesafterend}

\newcommand{\supp}{\operatorname{supp}}

\newcommand{\conj}{\operatorname{Conj}}
\renewcommand{\ker}{\operatorname{Ker}}

\newcommand{\cA}{{\mathcal A}}
\newcommand{\cT}{{\mathcal T}}
\newcommand{\Gr}{{\mathcal{K}_{\textswab G}}}

\newcommand{\cI}{\mathcal{I}}
\newcommand{\cB}{\mathcal{B}}
\newcommand{\Nat}{{\mathbb N}}
\newcommand{\norm}[1]{\Vert #1 \Vert}

\newcommand{\Ind}{{\bf I}}

\newcommand{\tF}{{\tau_F}}
\newcommand{\wSU}{\widehat{\operatorname{SU}}}
\newcommand{\SU}{{\operatorname{SU}}}

\newcommand{\bT}{\mathbb{T}}
\newcommand{\Zat}{\mathbb{Z}}
\newcommand{\ignore}[1]{{ }}
\newcommand{\m}{{\bf m}}

\begin{document}

%
%
%
%
%
%
%

\maketitle

\begin{abstract}
Weighted group algebras have been studied extensively in Abstract Harmonic Analysis where complete characterizations have been found for some important properties of weighted group algebras, namely amenability and Arens regularity. One of the generalizations of weighted group algebras   is weighted hypergroup algebras. Defining weighted hypergroups, analogous to weighted groups, we study   Arens regularity and isomorphism to operator algebras for them. We also examine our results on three classes of discrete weighted hypergroups constructed by conjugacy classes of FC groups,  the dual space of compact groups, and hypergroup structure defined by orthogonal polynomials. We observe some unexpected examples regarding Arens regularity and operator isomorphisms of weighted hypergroup algebras.

$\line(1,0){200}
$



\end{abstract}


%
%
%
%
%

\begin{section}{Introduction}\label{s:introduction}

Discrete hypergroups were defined as a generalization of (discrete) groups. Also,   some objects related to locally compact groups may be studied as discrete hypergroups.   For instance, double cosets of a locally compact group with respect to a compact open subgroup. In particular, this class includes the hypergroup structures on  conjugacy classes of an FC group (i.e. every conjugacy class is finite). Also, for a compact group $G$, the set of equivalence classes of irreducible (unitary) representations of $G$, denoted by $\widehat{G}$ and called the \ma{dual of the group $G$}, is a commutative  discrete hypergroup.
On one hand these examples as well as hypergroups  defined by orthogonal polynomials connect the studies done on hypergroups to different topics in abstract harmonic analysis.
On the other hand, the similarities of hypergroups and  groups suggest that one may be able to generalize the studies on  groups to hypergroups.

One of the topics related to hypergroups which has been initiated based on a similar study on  groups is \ma{weighted hypergroups} and \ma{weighted hypergroup algebras},   as they are defined in the following. The weighted hypergroup algebra, as a Banach algebra can be the subject of study for different properties of Banach algebras. The first studies over weighted hypergroup algebras may be tracked back to \cite{bh1, gh85, gh86}.

 In this manuscript, we study  Arens regularity and isomorphism to operator algebras for weighted hypergroup algebras.
To recall, the second dual of a Banach algebra can be equipped with two algebraic actions to form Banach algebras, we call a Banach algebra `Arens regular' if these two actions coincide.
Also a  Banach algebra $\cA$ is called an \ma{operator algebra} if there is a Hilbert space $\mathcal{H}$ such that $\cA$ is  a closed subalgebra of $\cB(\mathcal{ H})$.
The main result  of  \cite{ulger} rules out Arens regularity (and subsequently operator algebra isomorphism) of weighted hypergroup algebras for    non-discrete hypergroups. Consequently, this paper is only dedicated to discrete hypergroups, although many results proved in Sections~\ref{s:weight-on-hypergroups} hold for weights on  non-discrete hypergroups as well.
In this manuscript, we particularly  examine our results on various classes of  weighted hypergroup algebras  with respect to these properties.
 One may note that, for the specific weight $\omega\equiv 1$, the weighted case is reduced back to regular hypergroups and their algebras.

The paper is organized as follows.
We start this paper by  Section~\ref{s:discrete-hypergroups} wherein we give the definition of discrete hypergroups consistently  and briefly go through  three classes of  hypergroup structures we use in examples.
Section~\ref{s:weight-on-hypergroups} is devoted to weights on (discrete) hypergroups, their corresponding algebras, and their examples.
We continue this section by studying some examples.
In particular, in Subsection~\ref{s:weights-on-^G}, we introduce and study some hypergroup weights on the dual of compact groups.




Arens regularity of weighted group algebras has been studied by Craw and Young in \cite{yo}. They showed that a locally compact group $G$ has a weight $\omega$ such that $L^1(G,\omega)$ is Arens regular if and only if $G$ is discrete and countable. They also characterized the Arens regularity of weighted group algebras with respect to one feature of the (group) weight, called \ma{$0$-clusterness} as described in \cite{da2}.
In Section~\ref{s:Arens-regularity}, the Arens regularity of weighted hypergroup algebras for discrete hypergroups is studied and it is shown   (Theorem~\ref{t:Arens-regularity-summary}) that the strong $0$-clusterness of the  corresponding hypergroup weight results in the Arens regularity of the weighted hypergroup algebra (strong $0$-clusterness implies $0$-clusterness, \cite{da2}).

Injectivity and equivalently isomorphism of weighted group algebras to operator algebras have  been studied before, see \cite{sa3,va}.
In Section~\ref{s:Operator-algebra-weighted-hypergroups}, studying the hypergroup case, we demonstrate (Theorem~\ref{t:injectiv-l^1(H,omega)-for-T2-weights}) that for hypergroup weights which are weakly additive and whose inverse is $2$-summable over the hypergroup, the weighted hypergroup algebra is injective and hence an isomorphism to an operator algebra exists.
 To do so, we apply some results regarding Littlewood multipliers of hypergroups. This machinery lets us to examine a class of hypergroup weights which are not weakly  additive, namely exponential weights,  in Subsection~\ref{ss:OP-exponential-weights}.

In Sections~\ref{s:Arens-regularity} and \ref{s:Operator-algebra-weighted-hypergroups} we present many examples to highlight some unexpected contrasts with some results in the theory of   weighted Fourier algebras on compact groups (Examples~\ref{eg:Chebyshev-Arens}, \ref{eg:dual-weights-Arens-regular-but-not-the Beurling-Fourier-algebra} and  \ref{eg:chebyshev-OP}).

Some results of this paper were first presented in  the first author's Ph.D. thesis, \cite{ma-the}, under the supervision of Yemon Choi and Ebrahim Samei.
\end{section}

\begin{section}{Discrete hypergroups and  examples}\label{s:discrete-hypergroups}

In this paper, $H$ is always a discrete hypergroup in the sense of \cite{je} unless otherwise is stated. For basic definitions and facts we refer the reader to the fundamental paper of Jewett, \cite{je}, or the comprehensive book \cite{bl}.

\begin{subsection}{Definition}

Let $H$ be a discrete set.  Let $\ell^1(H)$ denote the  Banach space of all functions (bounded measures) $f:H\rightarrow \Bbb{C}$ which are absolutely summable with respect to the counting measure, i.e. $\norm{f}_1:=\sum_{x\in H} |f(x)| <\infty$.
Let $c_c(H)$ and $c_0(H)$   denote respectively  the space of all finitely supported and vanishing at infinity  elements of $\ell^\infty(H)$.
We call $H$  a \ma{discrete hypergroup} if the following conditions hold.
\begin{itemize}
\item[(H1)]{There exists an associative binary operation $*$ called \ma{convolution} on $\ell^1(H)$ under which $\ell^1(H)$ is a Banach algebra. Moreover,
for every $x$, $y$ in $H$, $\delta_x * \delta_y$ is a positive measure with a finite support and $\norm{\delta_x*\delta_y}_{\ell^1(H)}=1$.}
\item[(H2)]{ There exists an element (necessarily unique) $e$ in $H$ such that $
\delta_e * \delta_x = \delta_x * \delta_e = \delta_x$
for all $x$ in $H$.}
\item[(H3)]{There exists a (necessarily unique) bijection $x \rightarrow \check{x}$ of $H$ called \ma{involution} satisfying $
(\delta_x * \delta_y\check{)} = \delta_{\check{y}} * \delta_{\check{x}}$ { for all   } $ x, y \in H$.
}
\item[(H4)]{$e$ belongs to $\supp (\delta_x * \delta_y)$ if and only if $y = \check{x}$.}
\end{itemize}


We call a hypergroup $H$ \ma{commutative} if $\ell^1(H)$ forms a commutative algebra.
The \ma{left translation} on  $\ell^\infty(H)$  is defined by
$
L_xf:H\rightarrow \Bbb{C}$ where $L_xf(y):=f(\delta_{\tilde{x}}*\delta_y)$
for each $f$  in  $\ell^\infty(H)$ and $x,y\in H$.
A   non-zero, positive, left invariant linear functional
 $h$ (possibly unbounded)  on $c_c(H)$ is called a  \ma{Haar measure}, \cite{je}.
For a discrete hypergroup the existence of a Haar measure is proved and it is  unique up to multiplication by a positive constant.
Indeed, for  a discrete hypergroup,  a Haar measure $h:H\rightarrow (0,\infty)$ such that $h(e)=1$ is defined by
$h(x)=(\delta_{\check{x}}*\delta_x(e))^{-1}$ for all $x\in H$.

The \ma{hypergroup algebra}, denoted by $L^1(H,h)$   is  the Banach algebra of integrable functions on $H$ with respect to the Haar measure $h$ equipped with the convolution
 $f*_h g :=\sum_{x\in H} f(x) L_xg h(x)$.
It is easy to observe that  $f\mapsto fh$  is an isometric algebra isomorphism from the Banach algebra $L^1(H,h)$ onto the Banach algebra $\ell^1(H)$. Due to this isomorphism, we focus our study on $\ell^1(H)$ without loss of generality.

 \end{subsection}

 \begin{subsection}{The conjugacy classes of FC groups}
Let $G$ be a (discrete) group with the group algebra $\ell^1(G)$ and  $\conj(G)$ is the set of  all conjugacy classes of $G$.
We denote the centre of the group algebra by $Z\ell^1(G)$. The group $G$ is called an \ma{FC} or \ma{finite conjugacy group} if for each $C\in \conj(G)$, $|C|<\infty$. For such groups, $\conj(G)$ forms a commutative discrete hypergroup (which is the discrete case of the hypergroup structures defined in \cite[Subsection~8.3]{je}).
Let $\Psi$ denote the  linear mapping from $ Z\ell^1(G)$ to $\ell^1(\conj(G))$ defined by $\Psi(f)(C)=|C| f(C)$  for  $C\in\conj(G)$ where $\Psi(f)(C):=f(x)$ {for (every) $x\in C$.} Then  one can easily check that  $\Psi$ is an  isometric Banach algebra isomorphism between $\ell^1(\conj(G))$ and  $Z\ell^1(G)$.

As an extension of finite products of hypergroups (or in particular groups), let $\{H_i\}_{i\in\Ind}$ be a family of discrete hypergroups, then $H:=\bigoplus_{i\in \Ind} H_i$ where for each $x\in H$, $x=(x_i)_{i\in \Ind}$ where $x_i$ is the identity of the hypergroup $H_i$, $e_{H_i}$, for all $i\in \Ind$ except finitely many. $H$ is called \ma{restricted direct product} of $\{H_i\}_{i\in \Ind}$ which is a hypergroup (or a group if for every $i$, $H_i$ is a group).

\begin{eg}\label{eg:RDPF-groups}
For a family of FC groups $\{G_i\}_{i\in \Ind}$, let $G:= \bigoplus_{i\in\Ind}{ G_i}$ be the restricted direct product of $\{G_i\}_{i\in \Ind}$. Then $G$ is a discrete FC group and $\conj(G)$ is  the hypergroup generated by the restricted direct product of $\{\conj(G_i)\}_{i\in \Ind}$,
$\conj(G)=\bigoplus_{i\in \Ind} \conj(G_i)$.
\end{eg}

\end{subsection}

\begin{subsection}{The dual of compact groups}

Let $G$ be a compact group and  $\widehat{G}$ denotes the set of all irreducible unitary (necessary finite-dimensional) representations of a compact group $G$, up to unitary equivalence relation. It is known that the irreducible  decomposition of the tensor products of elements of $\widehat{G}$ leads to a discrete commutative hypergroup structure on $\widehat{G}$  where   $h(\pi)=d_{\pi}^2$ for $d_\pi$  the dimension of  $\pi\in \widehat{G}$. (See \cite[Example~1.1.14]{bl}).

\begin{eg}\label{eg:SU(2)}
Let $\operatorname{SU}(2)$ be the compact Lie group of $2\times 2$ special unitary matrices on $\Bbb{C}$, and let $\wSU(2)$ be the hypergroup of all irreducible representations on $\operatorname{SU}(2)$. It is known that $
\wSU(2)=(\pi_\ell)_{\ell \in \Nat_0}$
where $\Nat_0:=\{0,1,2,\cdots\}$ and  the dimension of $\pi_\ell$ is $\ell+1$. Moreover, for all $\ell, \ell'$, $\overline{\pi}_\ell=\pi_\ell$ and
$
\pi_\ell \otimes \pi_{\ell'} \cong  \pi_{|\ell-\ell'|} \oplus \pi_{|\ell-\ell'|+2 } \oplus \cdots \oplus \pi_{\ell+\ell'}$.
This tensor decomposition is called ``\ma{Clebsch-Gordan}'' decomposition formula. So using the  Clebsch-Gordan formula, we have that
\[
\delta_{\pi_\ell}*\delta_{\pi_{\ell'}} = \sideset{_2}{}\sum_{r=|\ell-\ell'|}^{\ell+\ell'} \frac{(r+1)}{(\ell+1)(\ell'+1)}\delta_{\pi_r}\]
where 
\begin{equation}\label{eq:sum-2}
\sideset{_2}{}\sum_{r=a}^b f(t) = f(a) + f(a+2) + \ldots + f(b-2) + f(b).
\end{equation}
Also
$
\overline{\pi}_\ell=\pi_\ell
$ and $h(\pi_\ell)=(\ell+1)^2$ for all $\ell$.
\end{eg}

\begin{eg}\label{eg:product-of-finite-groups}
Suppose that $\{G_i\}_{i\in\Ind}$ is a non-empty family of compact groups for arbitrary indexing set $\Ind$. Let  $G:=\prod_{i\in\Ind}G_i$ be the product of $\{G_i\}_{i\in \Ind}$ i.e. $G := \{(x_i)_{i\in \Ind}:\ x_i\in G_i\}$
equipped with the product topology. Then $\widehat{G}$ is
the restricted direct product of hypergroups $\{\widehat{G}_i\}_{i\in \Ind}$, (see \cite[Theorem~27.43]{he2}).
\end{eg}

\end{subsection}

\begin{subsection}{Polynomial hypergroups}


Let $\Nat_0=\Nat \cup \{0\}$.   Hypergroups related to systems of orthogonal polynomials in one variable have been introduced and studied by Lasser \cite{la*} and Voit \cite{voit*}.
Such a hypergroup structure on  $\Nat_0$ is called a \ma{polynomial hypergroup}  which is also discrete and commutative (\cite[Section~3.2]{bl}).
%
%


\begin{eg}\label{eg:basic-polynomial}
Let  $\Bbb{N}_0$ be equipped with the hypergroup convolution $\delta_n*\delta_m:=(1/2)\delta_{|n-m|}+ (1/2)\delta_{n+m}$. This hypergroup structure is called  \ma{Chebyshev polynomial of the first type}. One can show that the hypergroup algebra of $\Nat_0$ is isomorphic to the subalgebra of symmetric functions on the Fourier algebra of the torus, i.e. $Z_{\pm 1} A(\Bbb{T}):=\{f+\check{f}: f\in A(\Bbb{T})\}$ where  $\check{f}(x)=f(-x)$.
\end{eg}
\end{subsection}

\end{section}

\begin{section}{Weighted discrete hypergroups and examples}\label{s:weight-on-hypergroups}

In this section we study weights on discrete hypergroups, their corresponding algebras, and their examples. Specially we are interested to see concrete examples of weights defined on the classes of commutative discrete  hypergroups which were mentioned in Section~\ref{s:discrete-hypergroups}.

\begin{subsection}{General Theory}\label{ss:general-theory-weights}
We believe all the definitions and observations in this subsection still hold for non-discrete hypergroups (see \cite{gh85}), but here  we are mainly interested in the discrete case.

\begin{dfn}\label{d:weight-on-hypergroups}
Let $H$ be a discrete hypergroup. We call a function $\omega: H \rightarrow(0,\infty)$ a \ma{weight} if, for every $x,y\in H$, $\omega(\delta_x*\delta_y)\leq \omega(x)\omega(y)$.
Then we call $(H,\omega)$ a \ma{weighted hypergroup}. Let $\ell^1(H,\omega)$ be the set  of all complex functions on $H$ such that
\[
\norm{f}_{\ell^1(H,\omega)}:=\sum_{t\in H} |f(t)| \omega(t) <\infty.
\]
Then one can easily observe that $(\ell^1(H,\omega),\norm{\cdot}_{\ell^1(H,\omega)})$ equipped with the (extended) convolution of $\ell^1(H)$ forms a Banach algebra which is called a \ma{weighted hypergroup algebra}.
\begin{itemize}
\item{It is easy to see that if $\omega$ is a positive function on $H$ such that $\omega(t)\leq \omega(x)\omega(y)$ for all $t,x,y\in H$ where $t\in \supp(\delta_x*\delta_y)$, then $\omega$ is a weight on $H$. We call such a weight  a \ma{central weight}. We will show later that not all hypergroup weights are central. (See Examples~\ref{eg:weight-not- central} and \ref{eg:whights-which-are-not-central-weights})}
\item{A hypergroup weight $\omega$ on $H$ is called \ma{weakly additive}, if for some $C>0$,  $\omega(\delta_x*\delta_y)\leq C(\omega(x) + \omega(y)) $ for all $x,y\in H$.}
\item{Two weights $\omega_1$ and $\omega_2$ are called \ma{equivalent} if there are  constants $C_1, C_2$ such that $C_1\omega_1\leq \omega_2 \leq C_2\omega_2$.}
\end{itemize}
\end{dfn}



%
\begin{eg}\label{eg:product-of-hypergroups}
Let $\{H_i\}_{i\in\Ind}$ be a family of discrete hypergroups with corresponding  weights $\{\omega_i\}_{i\in \Ind}$ such that $\omega_i(e_{H_i})=1$ for all $i\in \Ind$ except finitely many. 
Then $\omega(x_i)_{i\in\Ind} :=\prod_{i\in\Ind} \omega_i(x_i)$ forms a hypergroup weight on the restricted direct product of hypergroups $\{H_i\}_{i\in \Ind}$.
\end{eg}



A discrete hypergroup $H$ is called \ma{finitely generated} if for a finite set $F \subseteq H$ with $F= \check{F}$, we have $H=\bigcup_{n\in \mathbb{N}} F^{*n}$ then $F$ is called  a \ma{finite symmetric generator} of $H$. We define
\begin{equation}\label{eq:tau}
\tF:H\rightarrow \Bbb{N} \cup \{0\}
\end{equation}
by $\tF(x):=\inf\{n \in \Bbb{N}: \ x\in F^{*n}\}$ for all $x\neq e$ and $\tF(e)=0$.
It is straightforward to verify that if $F'$ is another finite symmetric generator of $H$, then for some constants $C_1,C_2$,  $C_1\tau_{F'}\leq \tF \leq C_2\tau_{F'}$.
If there is no risk of confusion, we may just use $\tau$ instead of $\tF$.

\begin{dfn}\label{d:polynomial-exponential}
 For a given $\beta\geq 0$, $\omega_\beta(x):=(1+\tau(x))^\beta$   is a central weight on $H$ which is called a \ma{Polynomial weight}. Similarly, for given $C>0$ and $0\leq \alpha \leq 1$, $\sigma_{\alpha,C}(x):=e^{C\tau(x)^\alpha}$  is a central weight on $H$ which is called an \ma{Exponential weight}.
\end{dfn}


\begin{proposition}\label{p:weights-homomorphisms}
Let $H, H'$ be two discrete hypergroups and $\phi: H_1 \rightarrow H_2$ be a surjective hypergroup homomorphism. If $\omega$ is a weight on $H$ so that for every $x \in H$, $\omega(x)\geq \delta$ for some $\delta>0$. Then $\omega'$ defined   by
\[
{\omega}'(y):=\inf\{\omega(x):\ x\in H, \phi(x) =y\} \ \ \ \ \ \ \ (y\in H'),
\]
 is  a weight on $H'$.
\end{proposition}

\begin{proof}
Proof is immediate if one note that $\phi: c_c(H) \rightarrow c_c(H')$ satisfies $\norm{\phi(f)}_{\ell^1(H', \omega')} \leq \norm{f}_{\ell^1(H, \omega)}$ and
\[
\omega'(\delta_{\phi(x)} * \delta_{\phi(z)}) =\norm{  \delta_{\phi(x)} * \delta_{\phi(z)} }_{\ell^1(H', \omega')} =\norm{  \phi(\delta_{x}  * \delta_{z}) }_{\ell^1(H', \omega')} \leq \norm{  \delta_{x}    }_{\ell^1(H, \omega)}\norm{  \delta_{z}  }_{\ell^1(H, \omega)} = \omega(x) \omega(z)
\]
for every pair $x,z\in H$.
\end{proof}

\end{subsection}

\begin{subsection}{Weights on $\conj(G)$}\label{ss:weights-on-Conj(G)-from-G}

Let  $(G, \sigma)$ be a weighted group i.e. $\sigma(xy)\leq \sigma(x)\sigma(y)$ for all $x,y\in G$.   We use $\ell^1(G,\sigma)$ to denote the {weighted group algebra} constructed by $\sigma$.  Let  $Z\ell^1(G,\sigma)$   denote the center of  $\ell^1(G,\sigma)$. It is not hard to show that $Z\ell^1(G,\sigma)$ is the set of all $f\in \ell^1(G, \sigma)$ for them $f(yxy^{-1})=f(x)$  for all $x,y\in G$.

The following proposition lets us apply group weights  to generate hypergroup weights on $\conj(G)$. The proof is straightforward, so we omit it here.

\begin{proposition}\label{p:weights-coming-groups}\label{c:relation-between-Zl^1(G,weight)-and-Zl^1(conj(G),center-weight)}
Let $G$ be an FC group possessing  a  weight $\sigma$. Then the mean function $\omega_\sigma$ defined by  $\omega_\sigma(C):={|C|}^{-1}\sum_{t\in C} \sigma(t)$ {$(C\in \conj(G))$}
is a weight on the hypergroup $\conj(G)$.
Further,   $\ell^1(\conj(G),\omega_\sigma)$ is isometrically Banach algebra isomorphic to $Z\ell^1(G,\sigma)$.
\end{proposition}




\begin{rem}\label{r:central-weights-on-groups-UNIQUE}
Let $G$ be an FC group and let $\omega$ be a central weight on $\conj(G)$. Then the mapping $\sigma_\omega$, defined on $G$ by $\sigma_\omega(x):=\omega(C_x)$,  is  a group weight on $G$. And $\ell^1(\conj(G),\omega)$ as a Banach algebra is isometrically isomorphic to $Z\ell^1(G,\sigma_\omega)$.
\end{rem}

\begin{eg}\label{eg:|C|-is-a-weight-on-conj(G)}
Let $G$ be a discrete FC group. The mapping $\omega(C)=|C|$, for $C\in\conj(G)$, is a central weight on $\conj(G)$.
\end{eg}

\begin{eg}\label{eg:defined-I-C}
Let   $G =\bigoplus_{i\in\Ind}{ G_i}$ for a family of finite groups $\{G_i\}_{i\in \Ind}$.
 Given $C=(C_i)_{i\in \Ind}\in\conj(G)$, define $\Ind_C:=\{i\in \Ind: C_i\neq e_{G_i}\}$.
For each $\alpha>0$, we define a mapping  $\omega_\alpha(C):= \left(1 + |C_{{i_1}}| + \cdots + |C_{{i_n}}|\right)^\alpha$
 where $i_j \in \Ind_C$. We show that $\omega_\alpha$ is a central weight on $\conj(G)$.
 To do so let  $E \subseteq CD$ for some $E,C,D\in \conj(G)$. One can easily show that for each $i\in \Ind$, $E_i\subseteq C_i D_i$; $\Ind_E \subseteq \Ind_C \cup \Ind_D$.
 Therefore,
 \begin{eqnarray*}
 \omega_\alpha(C) &=& ( 1 + \sum_{i\in \Ind_E}|E_{i}|)^\alpha
 \leq (1 + \sum_{i\in \Ind_E} |C_{{i}}||D_{{i}}|)^\alpha\ \ \ \ \ \text{(by Example~\ref{eg:|C|-is-a-weight-on-conj(G)})}\\
 &\leq& \left( 1+ \sum_{i\in \Ind_C}|C_{{i}}|\right)^\alpha\;\left( 1+ \sum_{i\in \Ind_D}|D_{{i}}|\right)^\alpha= \omega_\alpha(C) \omega_\alpha(D).
 \end{eqnarray*}
  \end{eg}

  A group $G$ is called a group with \ma{finite  commutator group} or \ma{FD} if its derived subgroup is finite. It is immediate that for a  group $G$, for  every $C\in \conj(G)$,   $|C|\leq |G'|$ when $G'$ is the \ma{derived subgroup} of $G$. Therefore, the order of conjugacy classes of an FD group are uniformly bounded by $|G'|$. The converse is also true, that is for an FC group $G$, if the order of conjugacy classes are uniformly bounded, then $G$ is an FD group, see \cite[Theorem~14.5.11]{rob}.  The following proposition implies that every hypergroup weight on the conjugacy classes of an FD group which is constructed by a group weight (as given in Proposition~\ref{p:weights-coming-groups}) is equivalent to a central weight. We omit the proof of the following proposition as it is straightforward.

\begin{proposition}\label{p:Omega-is-a-weight}
Let $(G,\sigma)$ be a  weighted FD group.
 Then the hypergroup weight $\omega_z(C):=|G'|^2 \omega_\sigma(C)$, for $C\in\conj(G)$,
forms a central weight. Here $\omega_\sigma$ is defined as in Proposition~\ref{p:weights-coming-groups}.
\end{proposition}

In contrast to Proposition~\ref{p:Omega-is-a-weight}, we will see in the following examples that there exist  weights  on FC groups  (with infinite derived subgroup) which are not equivalent to any central weight.




\end{subsection}




\begin{eg}\label{eg:weight-not- central}
Let $S_3$ be the symmetric group of order $6$.
Let $\omega$ be defined on $\conj(S_3)$ by $\omega(C_e)=1$, $\omega(C_{(12)})=2$, and $\omega(C_{(123)})=5$.
One may verify that  $\omega$  is  a   weight on $\conj(S_3)$.
 On the other hand, since $5=\omega(C_{(123)}) \nleqslant \omega(C_{(12)})^2=4$, $\omega$ is not a central weight.
\end{eg}

\begin{eg}\label{eg:whights-which-are-not-central-weights}
We generate the restricted direct product  $G=\bigoplus_{n\in\Bbb{N}} S_3$. Let us define the weight $\omega' := \prod_{n\in \Bbb{N}}\omega$ on $\conj(G)$ where $\omega$ is the hypergroup weight on $\conj(S_3)$ defined in Example~\ref{eg:weight-not- central}.
For each $N\in\Bbb{N}$, define $D_N:=\prod_{n\in\Bbb{N}}D_n^{(N)} \in \conj(G)$ where $D_n^{(N)}=C_{(123)}$ for all $n\in 1, \ldots, N$ and $D_n^{(N)}=C_e$ otherwise. One can verify that $D_N \in \supp(\delta_{E_N}*\delta_{E_N})$ for $E_N=\prod_{n\in\Bbb{N}}E^{(N)}_n \in \conj(G)$ with $E_n^{(N)}=C_{(12)}$ for all $n\in 1, \ldots, N$ and $E_n^{(N)}=C_e$ otherwise. Therefore
\[
 \frac{\omega'(D_N)}{\omega'(E_N)^2}=\prod_{n=1}^{N} \frac{\omega(C_{(123)})}{\omega(C_{(12)})^2} = (5/4)^{N}  \rightarrow \infty
 \]
 where $N\rightarrow \infty$.  Hence, $\omega'$ is not equivalent to any central weight.
\end{eg}




We close this subsection with the following corollary of Lemma~\ref{p:weights-homomorphisms}.

\begin{cor}\label{c:Quotient-mapping-Tomega}
Let $G$ be an FC group, $N$  a normal subgroup of $G$, and $\omega$  a weight on $\conj(G)$ such that there is some  $\delta>0$ such that $\omega(C)>\delta$, for any $C\in \conj(G)$.
Then the mapping $\tilde{\omega}:\conj(G/N)\rightarrow \Bbb{R}^+$
defined by $\tilde{\omega}(C_{xN}):=\inf\{\omega(C_{xy}):\ y\in N\}$, for $C_{xN} \in \conj(G/N)$,
 forms  a weight on $\conj(G/N)$.
\end{cor}

\begin{subsection}{Weights on duals of compact groups}\label{s:weights-on-^G}

In this subsection, $G$ is a compact group. We recall that for each $\pi \in \widehat{G}$ and $f\in L^1(G)$,
\[
\widehat{f}(\pi):=\int_G f(x) \overline{\pi(x)} dx
\]
is the \ma{Fourier transform} of $f$ at $\pi$. Let $VN(G)$ denote the group von Neumann algebra of $G$, i.e. the von Neumann algebra generated by the left regular representation of $G$. It is well-known that the predual of $VN(G)$, denoted by $A(G)$, is a Banach algebra of continuous functions on $G$; it is called the \ma{Fourier algebra} of $G$. Moreover, for every $f\in A(G)$,
\[
\norm{f}:=\sum_{\pi \in \widehat{G}} d_\pi \norm{\widehat{f}(\pi)}_1 <\infty,
\]
where $\norm{\cdot}_1$ denotes the trace-class operator norm (look at \cite[Section~32]{he2}).

In an attempt to find the noncommutative analogue of weights on groups, Lee and Samei in \cite{LeSa} defined a \ma{weight on $A(G)$} to be a densely defined (not necessarily bounded) operator $W$ affiliated with $VN(G)$ and satisfying certain properties mentioned in \cite[Definition~2.4]{LeSa} (see also \cite{sp}).
Specially they assume that  $W$ has a bounded inverse, $W^{-1}$, which belongs to $VN(G)$.
For a weight $W$ on $A(G)$, the \ma{Beurling-Fourier algebra}  denoted by $A(G,W)$ is defined to be the set of all  $f\in A(G)$ such that
\[
\norm{f}_{A(G,W)}:=\sum_{\pi \in \widehat{G}} d_\pi \norm{\widehat{f}(\pi) \circ W}_1 <\infty.
\]
Indeed $(A(G,W), \|\cdot\|_{A(G,W)})$ forms a Banach algebra  with pointwise multiplication.  For abelian groups, the definition of Beurling-Fourier algebra corresponds the classical weighted group algebra  on the dual group. In \cite{LeSa}, the authors also studied Arens regularity and isomorphism to operator algebras for Beurling-Fourier algebras.

\begin{dfn}\label{d:defining-weights-on-^G}
Let $G$ be a compact group and $W$ a weight on $A(G)$. We define a function $\omega_W:\widehat{G} \rightarrow (0,\infty)$ by
\begin{equation}\label{eq:weights-on-^G}
\omega_W(\pi):=\frac{\norm{I_\pi\circ W}_1}{d_\pi}\ \ \ (\pi \in \widehat{G}),
\end{equation}
where $\norm{\cdot}_1$ denotes the trace norm and $I_\pi$ is the identity matrix corresponding to the Hilbert space of $\pi$.
\end{dfn}

As a specific class of weights on the Fourier algebra of a compact group $G$, in \cite{LeSa} (and independently, in \cite{sp}), \ma{central weights} on $A(G)$ are defined.
 Indeed, \cite[Theorem~2.12]{LeSa} implies that each central weight  $W$ can be represent by a unique function  $\omega_W:\widehat{G}\rightarrow (0,\infty)$ such that $\omega_W(\sigma) \leq \omega_W(\pi_1)\omega_W(\pi_2)$ for all $\pi_1,\pi_2,\sigma \in \widehat{G}$ where $\sigma \in \supp(\delta_{\pi_1}*\delta_{\pi_2})$. In this specific case of operator weights, $\omega_W$ matches with our definition in Definition~\ref{d:defining-weights-on-^G} for a central weight on the hypergroup $\widehat{G}$.   In the following we show that the same is true for a general   weight on $A(G)$ as well.

\vskip1.0em


Let us  define  $
ZA(G,W):=\{f\in A(G,W): f(yxy^{-1})=f(x)\ \text{for all $x\in G$}\}$
which is a Banach algebra with pointwise product and $\norm{\cdot}_{A(G,W)}$. Note that for the operator  weights $W$ where $\omega_W(\pi)=1$ for every $\pi \in \widehat{G}$,  $ZA(G,W)=ZA(G)$. For more on $ZA(G)$, look at \cite{zag}.

\begin{theorem}\label{t:ZA(G,W)=l^1(^G,w)}
Let $G$ be a compact group and $W$ a  weight on $A(G)$.
Then $\omega_W$ is a weight on the hypergroup $\widehat{G}$ and the weighted hypergroup algebra $\ell^1(\widehat{G},\omega_W)$ is isometrically isomorphic to $ZA(G,W)$.
\end{theorem}

\begin{proof}
Let ${\cal X}(G)$ denote the linear span of all the characters of $G$.
First  define a linear mapping $\cT: {\cal X}(G) \rightarrow c_c(\widehat{G})$  by
$\cT(\chi_\pi)=d_\pi\delta_\pi$ for each $\pi \in \widehat{G}$.  Let  $f=\sum_{i=1}^n \alpha_i \chi_{\pi_i} \in \mathcal{X}(G)$ for  $\pi_i \in \widehat{G}$ and $\alpha_i\in \Bbb{C}$.  In this case,
\begin{eqnarray*}
\norm{\cT(f)}_{\ell^1(\widehat{G},\omega)} &=& \sum_{i=1}^n |\alpha_i| d_{\pi_i} \omega(\pi_i) = \sum_{i=1}^n |\alpha_i| d_{\pi_i} \frac{\norm{I_{\pi_i}\circ W}_1}{d_{\pi_i}}\\
&=& \sum_{i=1}^n d_{\pi_i} \norm{\frac{\alpha_i}{d_{\pi_i}} I_{\pi_i}\circ W}_1 = \sum_{i=1}^n d_{\pi_i} \norm{ \alpha_i\widehat{\chi}_{\pi_i}(\pi_i)\circ W}_1=\norm{f}_{A(G,W)}.
\end{eqnarray*}
Therefore, $\cT$ forms a norm preserving linear mapping.
To show that ${\cal T}$ is an algebra homomorphism note that ${\cal T}(\chi_{\pi_1}\chi_{\pi_2})=   {\cal T}(\chi_{\pi_1})* {\cal T}(\chi_{\pi_2})$.
 It is known that ${\cal X}(G)$ is dense in $ZA(G,W)$ and clearly $c_c(\widehat{G})$ is dense in $\ell^1(\widehat{G}, \omega_W)$. So $\cT$ can be extended as an algebra isomorphism from $ZA(G,W)$ onto $\ell^1(\widehat{G},\omega_W)$ which preserves the norm. In particular,  $\ell^1(\widehat{G},\omega_W)$ forms an algebra with respect to its weighted norm and the convolution, and so $\omega_W$ is actually a hypergroup weight on $\widehat{G}$.
 \end{proof}

The proof of the following lemma is straightforward so we omit it here.

\begin{lemma}\label{l:dimension-as-a-weight}
Let $G$ be a compact group and $\widehat{G}$ be the set of all irreducible representations of $G$ as a discrete commutative hypergroup. Then $\omega_\beta(\pi)=d_\pi^\beta= h(d_\pi)^{\beta/2}$ is a central weight for each $\beta\geq 0$.
\end{lemma}


In the following, recall $ \sideset{_2}{}\sum$ defined in  (\ref{eq:sum-2}).

\begin{eg}\label{eg:weight-driven-from-Z-weights} {\bf \textsf{(Lifting weights from $\Bbb{Z}$ to $\widehat{\SU}(2)$)}}
Let $\sigma$ be a weight on the group $\Bbb{Z}$.
  We define
\begin{equation}\label{eq:omega-sigma}
\omega_\sigma(\pi_\ell) := \frac{1}{\ell+1} \sideset{_2}{}\sum_{r=-\ell}^{\ell} \sigma(r) \ \ \ \ \ (\ell \in \Nat_0).
\end{equation}
Recall that elements of $\wSU{(2)}$ can be regarded as $\pi_\ell$ when $\ell\in \Nat_0$.
Suppose that $m,n\in\Nat_0$ and without loss of generality $n\geq m$. Then,
\begin{eqnarray*}
\omega_\sigma(\pi_m)\omega_\sigma(\pi_n) &=& \frac{1}{m+1}\sideset{_2}{}\sum_{t=-m}^{m} \sigma(t) \; \frac{1}{n+1} \sideset{_2}{}\sum_{s=-n}^{n} \sigma(s)\\
 &\geq&
\frac{1}{(m+1)(n+1)}\sideset{_2}{}\sum_{t=-m}^{m} \sideset{_2}{}\sum_{s=-n}^{n} \sigma(t+s)
 \ \ \ \   \ \ \ \ \ \ \ \ \ \ \ \ \ (\dagger)\\
&=& \frac{1}{(m+1)(n+1)}\sideset{_2}{}\sum_{t=n-m}^{n+m} \sideset{_2}{}\sum_{s=-t}^t \sigma(s) \ \ \ \ \    \ \ \ \ \ \ \ \ \ \ \ \ \ \ \ (\ddagger)\\
&=& \sideset{_2}{}\sum_{t=n-m}^{n+m } \frac{(t+1)}{(m+1)(n+1)} \left(\frac{1}{t+1} \sideset{_2}{}\sum_{s=-t}^{t} \sigma(s) \right)\\
&=& \sideset{_2}{}\sum_{t=n-m}^{n+m} \frac{(t+1)}{(m+1)(n+1)} \omega_\sigma(\pi_t)\\
&=& \omega_\sigma(\delta_{\pi_m}*\delta_{\pi_n}).
\end{eqnarray*}
To show that the summations  $(\dagger)$ and $(\ddagger)$ are equal, let us arrange $(\dagger)$ as follows.
$$
\begin{array}{l l l l l}
\sigma(-m-n)&+\sigma(-m-n+2)&+\cdots&+\sigma(-m+n-2)&+\sigma(-m+n)\cr
          +\sigma(-m-n+2)&+\sigma(-m-n+4)&+\cdots&+\sigma(-m+n)&+\sigma(-m+n+2)\cr
           \vdots &  \vdots &  \ddots &  \vdots &  \vdots \cr
          +\sigma(m-2)&+\sigma(m)&+\cdots&+\sigma(m+n-4)&+\sigma(m+n-2)\cr
          +\sigma(m)&+\sigma(m+2)&+\cdots&+\sigma(m+n-2)&+\sigma(m+n)\ .\cr
          \end{array}$$
but the sum of  all the entries in the first column and the last row is equal to
$$\sideset{_2}{}\sum_{s=-m-n}^{m+n} \sigma(s)\ .$$
The next column and row give
$$\sideset{_2}{}\sum_{s=-m-n+2}^{m+n-2} \sigma(s)\ ,$$
and so on. So by doing this finitely many times, we get $(\ddagger)$.
Indeed,   weight $\omega_\sigma$ follows from the recipe of Definition~\ref{d:defining-weights-on-^G} using the non-central weight $W$ on $A(\SU(2))$ defined in (\ref{eq:W})  in Appendix~\ref{a:SU(2)}. So instead of the above computations, one also could use Theorem~\ref{t:ZA(G,W)=l^1(^G,w)} to prove that $\omega_\sigma$ is actually a weight on $\widehat{\SU}(2)$.
\end{eg}

\ignore{
\begin{eg}\label{eg:omega_a-on-SU(2)} {\bf \textsf{(Non-central weights on $\widehat{\SU}(2)$)}}
Let us define $\omega_a:\wSU(2)\rightarrow \Bbb{R}^+$ such that $\omega_a(\pi_\ell):={a^{\ell+1}}/({\ell+1})$
for a fixed constant  $a\geq {(\sqrt{5}+1)}/2$. We show that $\omega_a$ is a weight on $\wSU(2)$.
For a pair of $ \ell,\ell'$ in $\Bbb{N}_0:=\{0,1,2,3\ldots\}$, without loss of generality suppose that $\ell\geq \ell'$. So   we have
\begin{eqnarray*}
\sideset{_2}{}\sum_{r=\ell-\ell'}^{\ell+\ell'} (r+1) \omega_a(\pi_r) &=& \sideset{_2}{}\sum_{r=\ell-\ell'}^{\ell+\ell'} a^{r+1} =  a^{\ell-\ell'+1}\; \sideset{_2}{}\sum_{r=0}^{2\ell'} a^{r}\\
&=& a^{\ell-\ell'+1}\; \frac{a^{2\ell'  +2 }-1}{a^2 -1} =  \frac{a^{\ell + \ell' +3}}{a^2-1} -  \frac{a^{\ell - \ell' +1}}{a^2-1}\\
&\leq&  a^{\ell+\ell'+2} \left( \frac{a}{a^2-1}\right)  \leq \frac{a}{a^2 -1}\omega_a(\ell)(\ell+1)\omega_a(\ell')(\ell'+1).
\end{eqnarray*}
But since $a \geq {(\sqrt{5}+1)}/2$, $a/(a^2-1)\leq 1$; therefore, $\omega_a(\delta_{\pi_\ell}*\delta_{\pi_{\ell'}}) \leq \omega_a({\pi_\ell}) \omega_a({\pi_{\ell'}})$.
Note that
${\omega_a(\pi_{2\ell})}/{\omega_a(\pi_{\ell})^2}   \rightarrow \infty$
when $\ell\rightarrow \infty$; while $\pi_{2\ell}\in \supp(\delta_{\pi_\ell}*\delta_{\pi_\ell})$. Hence, not only is $\omega_a$  a non-central weight but also it is  not equivalent to any central weight.
\end{eg}

\begin{rem} Let $\sigma(n)=a^n$ for some $a\geq 1$ on $\Bbb{Z}$. Clearly, $\sigma$ is a weight on $\Bbb{Z}$ and therefore, one may consider the weight $\omega_\sigma$ as defined in Example~\ref{eg:weight-driven-from-Z-weights}. For each $\ell\in\Nat_0$ and $a\geq (1+\sqrt{5})/2$,
\begin{eqnarray*}
\omega_\sigma(\pi_\ell)= \frac{1}{\ell+1} \sideset{_2}{}\sum_{r=-\ell}^{\ell} a^r
                         = \frac{a^{-\ell}}{\ell+1}\; \frac{a^{2\ell+2}-1}{a^2-1}       =\omega_a(\pi_\ell)     \; \frac{a^{2\ell+2}-1}{a^{2\ell+1}(a^2-1)},
\end{eqnarray*}
where $\omega_a$  is the weight defined in  Examples~\ref{eg:omega_a-on-SU(2)}.
But   for every $\ell\in \Nat_0$ and $a>1$,
\[
\frac{1}{a+1} \leq  \frac{a^{2\ell+2} - 1}{a^{2\ell+1}(a^2-1)} \leq \frac{a}{a^2-1}.
\]
This implies that   the weights   $\omega_\sigma$ and $\omega_a$,  respectively in Examples~\ref{eg:weight-driven-from-Z-weights} and \ref{eg:omega_a-on-SU(2)}, are  equivalent.
\end{rem}
}

Fix $\beta>0$. One may apply the construction in Example~\ref{eg:weight-driven-from-Z-weights} for 
\begin{equation}\label{eq:sigma-beta}
\sigma(\ell):=\left\{
\begin{array}{l l}
1 & 0 \leq \ell\\
( 1 - \ell)^\beta & \ell < 0
\end{array}
\right.\ \ \ (\ell \in \Bbb{Z})
\end{equation}
  to construct a hypergroup weight $\omega_\sigma$ on $\wSU(2)$. Observe that the weight $\omega_\sigma$ is equivalent to the weight 
$\omega_\beta$
 defined in Lemma~\ref{l:dimension-as-a-weight}.
We will see in Section~\ref{s:Arens-regularity}, that this particular weight   will give interesting classes of examples.
To construct weights from subgroups of compact groups, one can look at \cite[Proposition~4.11]{sp}.

\end{subsection}

\begin{subsection}{Weights on polynomial hypergroups}\label{ss:weights-on-polynomials}

Recall that $\wSU(2)$ is a particular example of polynomial hypergroup so-called \ma{Chebyshev polynomials}. Similar arguments  can be applied to construct hypergroup weights on polynomial hypergroups applying group  weights of $\Bbb{Z}$.

\begin{eg}\label{eg:Chebyshev-weight-Arens}
 Let $f:\Bbb{N}_0 \rightarrow \Bbb{R}^+$ be an increasing function such that $f(0)=1$. 
Then $\omega_f(n)=f(n)+2$  is a  central weight on $\Bbb{N}_0$ when it is equipped with the Chebyshev polynomial hypergroup structure of the first type.
Applying the argument  in Example~\ref{eg:basic-polynomial}, we can see that $\ell^1(\Bbb{N}_0,\omega_f)$ is  isomorphic to the symmetric subalgebra of $A(\Bbb{T},\sigma)$, that is $Z_{\pm 1}A(\Bbb{T},\sigma):=\{f+\check{f}: f\in A(\Bbb{T},\sigma)\}$, for  the group weight
\[
\sigma_f(\ell):=\left\{
\begin{array}{l l}
1 & \ell \geq 0\\
f(-\ell) & \ell<0
\end{array}
\right. \ \ \ (\ell\in \Bbb{Z})
\]
\end{eg}

\end{subsection}
\end{section}

\begin{section}{Arens regularity}\label{s:Arens-regularity}


In \cite[Chaptetr~4]{kam}, Kamyabi-Gol applied the topological center of  hypergroup algebras to prove some results about the hypergroup algebras and their second duals. For example, in \cite[Corollary~4.27]{kam}, he showed that for a (not necessarily discrete and commutative) hypergroup  $H$ (which possesses a Haar measure $h$), $L^1(H,h)$ is Arens regular if and only if $H$ is  finite.

Arens regularity of weighted group algebras has been studied by Craw and Young in \cite{yo}. They showed that a locally compact group $G$ has a weight $\omega$ such that $L^1(G,\omega)$ is Arens regular if and only if $G$ is discrete and countable.  The monograph \cite{da2} presents a thorough report on Arens regularity of weighted group algebras. In the following we adapt the machinery developed in \cite[Section~8]{da2} for weighted hypergroups. In \cite[Section~3]{da2}, the authors  study repeated limit conditions and give a rich  variety of results for them. Here, we will use some of them.

First  let us recall the following definitions.
Let $\cA$ be a Banach algebra. For  $f,g\in \cA$, $\phi\in \cA^*$,  and $F,G\in \cA^{**}$, we define the following module actions.
\[
\begin{array}{l l}
\langle f\cdot \phi, g\rangle:=\langle \phi, gf\rangle,& \langle \phi \cdot f, g\rangle:= \langle \phi, fg\rangle\\
\langle \phi\cdot F, f\rangle := \langle F, f \cdot \phi \rangle, & \langle F\cdot \phi, f\rangle:=\langle F, \phi\cdot f\rangle\\
\langle F\Diamond G,  \phi\rangle:=\langle G, \phi\cdot F\rangle, & \langle  G \Box F, \phi \rangle := \langle G, F \cdot \phi \rangle.
\end{array}
\]
Let  $F,G\in {\mathcal A}^{**}$, and let  $(f_\alpha)_\alpha$ and $(g_\beta)_\beta$ be nets in $\mathcal A$ such that $f_\alpha\rightarrow F$ and $g_\beta \rightarrow G$ in the weak$^*$ topology.
One may show that for products $\Box$ and $\Diamond$ of $\mathcal A^{**}$,
\[
F\Box G = w^*-\lim_\alpha w^*-\lim_\beta f_\alpha g_\beta \ \ \text{and} \ \ \ F\Diamond G=  w^*-\lim_\beta w^*-\lim_\alpha f_\alpha g_\beta.
\]
The Banach space $\cA^{**}$ equipped with either of the multiplications $\Box$ or   $\Diamond$ forms a Banach algebra.
The Banach algebra $\cA$ is called \ma{Arens regular} if two actions $\Box$ and $\Diamond$ coincide.


Let $c_0(H,\omega^{-1}):=\{ f:H\rightarrow \Bbb{C}:\ \ f\omega^{-1}\in c_0(H)\}$.
Note that $\ell^1(H,\omega)$ is the dual of $c_0(H,\omega^{-1})$. Hence, $\ell^1(H,\omega)^{**}$ can be decomposed as $\ell^1(H,\omega)\bigoplus c_0(H,\omega^{-1})^\perp$
when   $c_0(H,\omega^{-1})^\perp:=\{F\in \ell^1(H,\omega)^{**}:\ \langle F,\phi\rangle =0 \; \text{for all $\phi \in c_0(H,\omega^{-1})$}\}$.
To see this decomposition, let $F\in \ell^1(H,\omega)^{**}$, it is clear that $f:=F|_{c_0(H,\omega^{-1})}\in \ell^1(H,\omega)$ and consequently $\Phi:=F-f \in c_0(H,\omega^{-1})^\perp$. Therefore, $F=(f,\Phi) \in \ell^1(H,\omega)\bigoplus c_0(H,\omega^{-1})^\perp$.

\begin{proposition}\label{p:Arens-regularity}
Let $(H,\omega)$ be a weighted hypergroup. Then $\ell^1(H,\omega)$ is Arens regular if  the multiplications $\Box$ and $\Diamond$ restricted  to $c_0(H,\omega^{-1})^{\perp}$  are constantly   $0$.
\end{proposition}

\begin{proof}
Now let $F=(f,\Phi)$ and $G=(g,\Psi)$ belong to $\ell^1(H,\omega)^{**}$.
First, note that   $f\Box \Psi=f\Diamond \Psi$ and  $\Phi \Box g=\Phi \Diamond g$.
Thus $F \Box G =  (f, \Phi)\Box (g,\Psi) = (fg, f\Box \Psi+ \Phi \Box g) = (fg, f\Diamond \Psi + \Phi \Diamond g) = F\Diamond G$.
\end{proof}




Let us define the bounded function $\Omega_\omega: H\times H \rightarrow (0,1]$ by
\begin{equation}\label{eq:Omega}
\Omega_\omega(x,y):=\frac{\omega(\delta_x*\delta_y)}{\omega(x)\omega(y)}\ \ \ (x,y\in H).
\end{equation}
If there is no risk of confusion, we may use $\Omega$ instead of $\Omega_\omega$.

\vskip1.0em

 For a weighted group $(G,\sigma)$, the Arens regularity of weighted group algebras has been characterized completely; \cite[Theorem~8.11]{da2} proves that it is equivalent to the \ma{$0$-clusterness} of the function $\Omega_\sigma$ on $G\times G$, that is
\[
\lim_n \lim_m \Omega_\sigma(x_m, y_n) = \lim_m \lim_n \Omega_\sigma(x_m, y_n)=0
\]
whenever $(x_m)$ and $(y_n)$ are sequences in $G$, each consisting of distinct points, and both repeated limits exist.
A stronger version of  $0$-clusterness is called \ma{strong  $0$-clusterness} (see \cite[Section~3]{da2}).
We define strongly $0$-cluster functions as presented in   \cite[Definition~3.6]{da2} for  discrete topological spaces.

\begin{dfn}
Let $X$ and $Y$ be two sets and   $f$ is a bounded   function on $X\times Y$ into $\Bbb{C}$. Then  $f$ \ma{$0$-clusters strongly} on $X\times Y$ if
\[
\lim_{x\rightarrow\infty} \limsup_{y\rightarrow\infty} f(x,y) = \lim_{y\rightarrow\infty} \limsup_{x\rightarrow \infty} f(x,y) = 0.
\]
\end{dfn}


\vskip1.5em
Let us define   Banach space isomorphism $\kappa: \ell^1(H,\omega)\rightarrow \ell^1(H)$ where $\kappa(f)=f\omega$ for each $f\in \ell^1(H,\omega)$.
Note that for $\kappa^{**}:\ell^1(H,\omega)^{**}\rightarrow \ell^1(H)^{**}$ and $\Phi\in c_0(H,\omega)^\perp$, one gets
$
 \langle \kappa^{**}(\Phi), \phi\rangle = \langle \Phi, \kappa^*(\phi)\rangle$
 which is   $0$ for all $\phi\in c_0(H)$. Therefore $\kappa^{**}(\Phi)\in c_0(H)^\perp$. The converse (which we do not use here) is also true and straightforward to show.

\vskip1.0em

The following theorem is a generalization of \cite[Theorem~8.8]{da2}. In the proof we use some techniques of
the proof of \cite[Theorem~3.16]{LeSa}.

\begin{theorem}\label{t:Box-product-is-zero}
Let $(H,\omega)$ be a weighted hypergroup and let $\Omega$  $0$-cluster strongly on $H\times H$. Then $\Phi\Box \Psi=0$ and $\Phi\Diamond \Psi=0$  whenever $\Phi,\Psi\in c_0(H,1/\omega)^{\perp}$.
\end{theorem}

\begin{proof}
Let us show the theorem for $\Phi\Box \Psi$, the proof for the other action is  similar. Let $\Phi,\Psi\in c_0(H,1/\omega)^\perp$. By   Goldstine's theorem, there are nets $(f_\alpha)_\alpha, (g_\beta)_\beta \subseteq \ell^1(H)$  such that $f_\alpha \rightarrow \kappa^{**}(\Phi)$ and $g_\beta \rightarrow \kappa^{**}(\Psi)$ in the weak$^*$ topology of $\ell^1(H)^{**}$ while  $\sup_\alpha\norm{f_\alpha}_1\leq 1$ and $\sup_\beta\norm{g_\beta}_1\leq 1$.
So for each $\psi\in \ell^\infty(H)$ and $\Phi, \Psi\in \ell^1(H,\omega)^{**}$,
\begin{eqnarray*}
\langle\psi\omega, \kappa^{**}(\Phi\Box \Psi)\rangle= \langle \kappa^*(\psi), \Phi\Box \Psi \rangle=  \lim_\alpha  \lim_\beta \langle \psi\omega, \kappa^{-1}(f_\alpha) * \kappa^{-1}(g_\beta) \rangle
=  \lim_\alpha  \lim_\beta \langle \psi\omega, f_\alpha/\omega * g_\beta/\omega \rangle.
\end{eqnarray*}
Thus
\begin{eqnarray*}
|\langle\psi\omega, \kappa^{**}(\Phi\Box \Psi)\rangle|
 &=& \lim_\alpha \lim_\beta |\langle \psi\omega, f_\alpha/\omega * g_\beta/\omega\rangle| \\
&=& \lim_\alpha \lim_\beta  \left|\sum_{y\in H} \psi(y) \omega(y) \sum_{x,z\in H} \frac{f_\alpha(x)}{\omega(x)} \frac{g_\beta(z)}{\omega(z)} \delta_x*\delta_z(y)\right|\\
&\leq & \limsup_\alpha \limsup_\beta   \sum_{x,z\in H} \frac{|f_\alpha(x)|}{\omega(x)} \frac{|g_\beta(z)|}{\omega(z)} \sum_{y\in H} |\psi(y)| \omega(y) \delta_x*\delta_z(y)\\
&\leq & \norm{\psi}_{\ell^\infty(H)}\; \limsup_\alpha \limsup_\beta    \sum_{x,z\in H} |f_\alpha(x)||g_\beta(z)| \sum_{y\in H}  \frac{\omega(y)}{\omega(x)\omega(z)} \delta_x*\delta_z(y)\\
&=&\norm{\psi}_{\ell^\infty(H)}\;  \limsup_\alpha \limsup_\beta   \sum_{x,z\in H} |f_\alpha(x)||g_\beta(z)|  \Omega(x,z).
\end{eqnarray*}

For a given $\epsilon>0$, since by the hypothesis $\lim_{x}\limsup_{z}\Omega(x,z)=0$, there is a finite set $A\subseteq H$ such that for each $x\in A^c( =H\setminus A)$ there exists a finite set $B_{x}\subseteq H$ such that for each $z\in B_{x}^c:=H\setminus B$, $|\Omega(x,z)|\leq \epsilon$. 
First note that
\begin{eqnarray*}
\limsup_\alpha \limsup_\beta    \sum_{x\in A^c}\sum_{z\in B_x^c} |f_\alpha(x)||g_\beta(z)|  \Omega(x,z)\leq \limsup_\alpha \limsup_\beta \epsilon \norm{f_\alpha}_1 \norm{g_\beta}_1 \leq \epsilon.
\end{eqnarray*}
Also according to our assumption about $\Phi$ and $\Psi$ and since for each  $x\in H$, $\delta_x\in c_0(H,1/\omega)$, $\lim_\alpha f_\alpha(x) =0$ and  $\lim_\beta g_\beta(x) =0$.
So for the given $\epsilon>0$, there is  $\alpha_0$  such that for all $\alpha_0 \preccurlyeq \alpha$, $|f_\alpha(x)|<\epsilon/|A|$ for all $x\in A$. Moreover, for each $x\in A^c$ there is some $\beta_0^x$ such that for all $\beta$ where $\beta_0^x \preccurlyeq \beta$, $|g_\beta(z)|<\epsilon/|B_x|$ for all $z\in B_x$ (this is possible since $A$ and $B_x$ are finite). Therefore, since $|\Omega(x,z)|\leq 1$,
\begin{eqnarray*}
\limsup_\alpha \limsup_\beta    \sum_{x\in A}\sum_{z\in H} |f_\alpha(x)||g_\beta(z)|  \Omega(x,z)\leq  \limsup_\beta  \epsilon \norm{g_\beta}_1 = \epsilon
\end{eqnarray*}
and
\begin{eqnarray*}
\limsup_\alpha \limsup_\beta    \sum_{x\in A^c}\sum_{z\in B_x} |f_\alpha(x)||g_\beta(z)|  \Omega(x,z)
&\leq&
\limsup_\alpha   \sum_{x\in A^c} |f_\alpha(x)| \limsup_\beta  \sum_{z\in B_x} |g_\beta(z)|\\
&\leq&
  \limsup_\alpha  \epsilon \norm{f_\alpha}_1 = \epsilon.
\end{eqnarray*}
But
\begin{eqnarray*}
\sum_{x,z\in H} |f_\alpha(x)||g_\beta(z)|  \Omega(x,z) &=&
 \sum_{x\in A^c, z\in B_x^c} |f_\alpha(x)||g_\beta(z)|  \Omega(x,z)\\
 &+& \sum_{x\in A, z\in H} |f_\alpha(x)||g_\beta(z)|  \Omega(x,z)\\
&+&\sum_{x\in A^c, z\in B_x} |f_\alpha(x)||g_\beta(z)|  \Omega(x,z),
\end{eqnarray*}
and so, one gets that $|\langle\psi\omega, \kappa^{**}(\Phi\Box \Psi)\rangle| \leq 3 \epsilon \norm{\psi}_\infty$. Since $\epsilon>0$ was arbitrary, this proves the claim of the theorem.
  \end{proof}

\begin{theorem}\label{t:Arens-regularity-summary}
Let $(H, \omega)$ be a discrete weighted hypergroup and consider the following conditions:
\begin{enumerate}
\item[$(1)$]{ $\Omega$ $0$-clusters strongly on $H\times H$.}
\item[$(2)$]{ $\Phi \Box \Psi = \Phi \Diamond \Psi =0 $ for all $\Phi, \Psi \in c_0(H, 1/\omega)^{\perp}$.}
\item[$(3)$]{ $\ell^1(H,\omega)$ is Arens regular.}
\end{enumerate}
Then $(1) \Rightarrow (2) \Rightarrow (3)$.
\end{theorem}

\begin{proof}
$(1) \Rightarrow (2)$ by Theorem~\ref{t:Box-product-is-zero}. $(2) \Rightarrow (3)$ is implied from Proposition~\ref{p:Arens-regularity}.
\end{proof}

\begin{rem}
Since in hypergroups, the cancellation does not necessarily exist, the argument of \cite[Theorem~1]{yo} cannot be applied to show $(3)$ implies $(1)$.
\end{rem}

\begin{eg}\label{eg:Chebyshev-Arens}
Let $\Nat_0$ be equipped with Chebyshev polynomial hypergroup structure of the first type and $\sigma_f$ be the group weight defined in Example~\ref{eg:Chebyshev-weight-Arens} for an increasing function $f$. One can easily check that if $\lim_{n,m}  f(n+m)/{f(n)f(m)}=0$,
then $\Omega_{\omega_f}$ $0$-clusters strongly on $\Nat_0\times \Nat_0$; hence, $\ell^1(\Nat_0,\omega_f)$ is Arens regular. Indeed,  $Z_{\pm}A(\Bbb{T},\sigma_f)$ is Arens regular. But note that $A(\Bbb{T},\sigma_f)$ (which is isomorphic to $\ell^1(\Bbb{Z},\sigma_f)$ through the Fourier transform) is not Arens regular, as $\Omega_{\sigma_f}$ does not $0$-cluster strongly on $\Bbb{Z}\times \Bbb{Z}$ (see \cite[Theorem~8.11]{da2}).
\end{eg}





\begin{cor}\label{c:Arens-subadditive}
Let $(H,\omega)$ be a weighted discrete hypergroup such that $\omega$ is a weakly additive weight. If $1/\omega \in c_0(H)$, then $\ell^1(H,\omega)$ is Arens regular.
\end{cor}

\begin{proof}
We have
\begin{eqnarray*}
\lim_{x\rightarrow \infty}\limsup_{y\rightarrow\infty} \frac{\omega(\delta_{x}*\delta_{y})}{\omega(x)\omega(y)}
&\leq &  \limsup_{x\rightarrow \infty}\limsup_{y\rightarrow\infty} C \frac{\omega(x)+ \omega(y)}{\omega(x)\omega(y)}\\
&=& C \limsup_{x\rightarrow \infty}\limsup_{y\rightarrow\infty} \frac{1}{\omega(x)}+\frac{1}{\omega(y)}=0.
\end{eqnarray*}
Therefore $\Omega$ $0$-clusters strongly on $H\times H$ and hence $\ell^1(H,\omega)$ is Arens regular by Theorem~\ref{t:Arens-regularity-summary}.
\end{proof}

\begin{corollary}\label{c:Arens-of-finite-generated}
Let $H$ be a  finitely generated hypergroup. Then for each polynomial weight $\omega_\beta$ ($\beta>0$)  on $H$ defined in Definition~\ref{d:polynomial-exponential},  $\ell^1(H,\omega_\beta)$ is Arens regular.
\end{corollary}

\begin{proof}
Here we only need to prove the case for an infinite hypergroup $H$.
Let $F$ be a finite generator of the hypergroup $H$ containing the identity of $H$ rendering the central weight $\omega_\beta$. Recall that $\omega_\beta$ is  weakly additive with constant $C=\min\{1,2^{\beta-1}\}$. Moreover, for each $N\in \Bbb{N}$,  for   $x\in H\setminus F^{*N}$, $\tF(x)\geq N$; hence,
$\omega_\beta(x)=(1+\tF(x))^\beta \geq (1+N)^\beta$.
Therefore, $1/\omega_\beta \in c_0(H)$. Subsequently,   $\ell^1(H,\omega_\beta)$ is Arens regular, by Corollary~\ref{c:Arens-subadditive}.
\end{proof}

\begin{rem}\label{r:central-screws}
Every  finitely generated hypergroup $H$ admits a weight  for which the corresponding  weighted algebra is Arens regular. On the other hand, an argument similar to \cite[Corollary~1]{yo} may apply to show that for every uncountable discrete hypergroup $H$, $H$ does not have any  weight $\omega$ which $0$-clusters strongly.
\end{rem}

\begin{eg}\label{eg:dual-weights-Arens-regular-but-not-the Beurling-Fourier-algebra}
Let $\omega_\beta$ be as defined in Lemma~\ref{l:dimension-as-a-weight} for some $\beta\geq 0$. Then     $\Omega_{\omega_\beta}$ also $0$-clusters strongly on $\wSU(2) \times \wSU(2)$. Therefore,  $\ell^1(\wSU(2),\omega_\beta)$, which is isometrically Banach algebra isomorphic to $ZA(\SU(2), \omega_\beta)$, is Arens regular. On the other hand, $A(\SU(2),\omega_\beta)$ is not Arens regular if $\beta>0$. To observe the later fact, first note that by applying  \cite{yo},  we obtain that $\ell^1(\Bbb{Z},\sigma)$ is not Arens regular for  $\sigma$ defined in (\ref{eq:sigma-beta}). 
 Therefore, $A(\Bbb{T},\sigma)$ is not Arens regular. Note that, $\omega_\beta$ can also  be rendered using the weight $\sigma$ through the argument of the last paragraph of Subsection~\ref{s:weights-on-^G}. 
For the dual spaces $VN(\Bbb{T},\sigma)$ and $VN(\SU(2), \omega_\beta)$, one may verify that $VN(\Bbb{T},\sigma)$ embeds $*$-weakly in $VN(\SU(2),\omega_\beta)$ (the details of this embedding will appear in a manuscript by the second named author and et al). Hence, $A(\Bbb{T},\sigma)$ is a quotient  of $A(\SU(2),\omega_\beta)$ and consequently $A(\SU(2),\omega_\beta)$ is not Arens regular.
\end{eg}









In the following, we  generalize some results on $\SU(2)$ to all $\SU(n)$'s, the group of all $n\times n$ special unitary matrices on $\Bbb{C}$,  based on a recent study on the representation theory of $\SU(n)$, \cite{lee1}.
As an example for Lemma~\ref{l:dimension-as-a-weight}, $(\wSU(n),\omega_\beta)$ is a discrete commutative hypergroup where $\omega_\beta(\pi)=d_\pi^\beta$ for some $\beta\geq 0$. See \cite{ful} for the details of representation theory of $\SU(n)$.
There is a one-to-one correspondence between $\wSU(n)$ and $n$-tuples $(\pi_1,\ldots,\pi_n)\in \Nat_0^n$ such that
$\pi_1 \geq \pi_2 \geq \cdots \geq \pi_{n-1} \geq \pi_n=0$.
This presentation of the representation theory of $\SU(n)$ is called \ma{dominant weight}. Using this presentation, we have the following formula which gives the dimension of each representation by the formula
\begin{equation}\label{eq:dimention-in-SU(n)}
d_{\pi}=\prod_{1\leq i < j \leq n} \frac{\pi_i -\pi_j + j -i}{j-i}
\end{equation}
where $\pi$ is the representation corresponding to $(\pi_1,\ldots,\pi_n)$. Suppose that $\pi, \nu,\mu$ are representations corresponding to $(\pi_1,\ldots,\pi_n)$, $(\nu_1,\ldots, \nu_n)$, and $(\mu_1,\ldots,\mu_n)$, respectively, such that $\pi \in\supp(\delta_\nu * \delta_\mu)$.   Collins, Lee, and \`{S}niady showed in \cite[Corollary~1.2]{lee1} there exists some $C_n>0$, for each $n\in\Bbb{N}$, such that
\begin{equation}\label{eq:Lee-inequality}
\frac{d_\pi}{d_\mu d_\nu} \leq  C_n \left(\frac{1}{1+\mu_1}  + \frac{1}{1+\nu_1}\right).
\end{equation}
  Applying (\ref{eq:Lee-inequality}), we prove that $\omega_\beta$ $0$-clusters on $\wSU(n)$.

\begin{proposition}\label{p:Arens-of-l1(SU(n))}
For every $\beta>0$, $\ell^1(\wSU(n),\omega_\beta)$ is Arens regular.
\end{proposition}

\begin{proof}
Let $(\mu_m)_{m\in \Bbb{N}}$ and $(\nu_k)_{k\in\Bbb{N}}$ be two arbitrary sequences of distinct elements of $\wSU(n)$.
Since, the elements of $(\mu_m)_{m\in \Bbb{N}}$  are distinct, $\lim_{m\rightarrow\infty} \mu_1^{(m)}=\infty$   where
$\mu_m=(\mu_1^{(m)},\ldots,\mu_n^{(m)})$. The very same thing can be said for $\nu_k=(\nu_1^{(k)},\ldots,\nu_n^{(k)})$.
For each arbitrary pair $(m,k)\in \Bbb{N}\times \Bbb{N}$, if $\pi\in\supp(\delta_{\mu_m}*\delta_{\nu_k})$, we have
\[
d_\pi\leq C_n (\frac{1}{1+\mu^{(m)}_1}+\frac{1}{1+\nu^{(k)}_1})d_{\mu_m}d_{\nu_k}.
\]
Hence
\[
\omega_\beta(\pi) \leq C_n^\beta  (\frac{1}{1+\mu^{(m)}_1}+\frac{1}{1+\nu^{(k)}_1})^\beta \omega_\beta(\mu_m)\omega_\beta(\nu_k).
\]
Therefore
\begin{eqnarray*}
\omega_\beta(\delta_{\mu_m}*\delta_{\nu_k}) = \sum_{\pi\in\wSU(n)}\delta_{\mu_m}*\delta_{\nu_k}(\pi)\omega_\beta(\pi)
\leq C_n^\beta  (\frac{1}{1+\mu^{(m)}_1}+\frac{1}{1+\nu^{(k)}_1})^\beta \omega_\beta(\mu_m)\omega_\beta(\nu_k).
\end{eqnarray*}
Or equivalently
\[
\Omega_\beta(\mu_m,\nu_k):=\frac{\omega_\beta(\delta_{\mu_m}*\delta_{\nu_k})}{\omega_\beta(\mu_m)\omega_\beta(\nu_k)} \leq   C_n^\beta  (\frac{1}{1+\mu^{(m)}_1}+\frac{1}{1+\nu^{(k)}_1})^\beta.
\]
Hence, $\lim_{m\rightarrow\infty}\limsup_{k\rightarrow\infty} \Omega_\beta(\mu_m,\nu_k)=
\lim_{k\rightarrow\infty}\limsup_{m\rightarrow\infty} \Omega_\beta(\mu_m,\nu_k) = 0$.
Since $\wSU(n)$ is countable, this argument implies that  $\Omega_\beta$ $0$-clusters strongly on $\wSU(n)\times \wSU(n)$ and, by Theorem~\ref{t:Arens-regularity-summary}, $\ell^1(\wSU(n),\omega_\beta)$ is Arens regular.
\end{proof}



\begin{eg}\label{eg:SL(2,2n)-Arens-regularity}
Let $SL(2,2^n)$ denote the finite group of special linear matrices
over the field $\Bbb{F}_{2^n}$ with cardinal $2^n$, for given $n\in\Bbb{N}$.
As a direct result of the character table, \cite{SL(2F)}, for each three conjugacy classes say $C_1,C_2,D\in\conj(SL(2,2^n))$, $|D|\leq 2( |C_1| + |C_2|)$ if $D\subseteq C_1C_2$ for all $n$.  Let us define the FC group $G$ to be the restricted direct product of $\{SL(2,2^n)\}_{n\in \Bbb{N}}$ i.e. $G:= \bigoplus_{n=1}^\infty SL(2,2^n)$.
Therefore,  one can easily show that   the weight $\omega_\alpha$, defined  in Example~\ref{eg:defined-I-C}, is a weakly additive weight with the constant $M=2^\alpha\min\{1, 2^{\alpha-1}\}$. Moreover, since $\lim_{C\rightarrow \infty} \omega_\alpha(C)=\infty$, $\ell^1(\conj(G),\omega_\alpha)$ is Arens regular, by Corollary~\ref{c:Arens-subadditive}.
\end{eg}

 \begin{rem}
Let $\omega$ be a central weight on $\conj(G)$ for some FC group $G$. Then there is a group weight $\sigma_\omega$, as defined in Remark~\ref{r:central-weights-on-groups-UNIQUE}, such that $\ell^1(\conj(G), \omega)$ is isometrically Banach algebra isomorphic to $Z\ell^1(G, \sigma_\omega)$.
So one may also use the embedding $\ell^1(\conj(G), \omega) \hookrightarrow \ell^1(G,\sigma_\omega)$ to study Example~\ref{eg:SL(2,2n)-Arens-regularity} by  applying the theorems which are characterizing Arens regularity of weighted group algebras.
\end{rem}

\begin{rem}\label{r:Arens-of-Qutient}
Let $G$ be an FC group and $\sigma$ a group weight on $G$. We defined $\omega_\sigma$, the derived weight on $\conj(G)$ from $\sigma$ in Proposition~\ref{p:weights-coming-groups}. Recall that in this case $Z\ell^1(G, \sigma)$ is isomorphic to the Banach algebra $\ell^1(\conj(G), \omega_\sigma)$. If $N$ is a normal subgroup of $G$,  we defined a quotient mapping $T_{\omega_\sigma}: \ell^1(\conj(G),\omega_\sigma) \rightarrow \ell^1(\conj(G/N),\tilde{\omega}_\sigma)$ in Corollary~\ref{c:Quotient-mapping-Tomega} where  $\tilde{\omega}_\sigma(C_{xN})=\inf\{\omega_\sigma(C_{xy}):\ y\in N\}\ \ (C_{xN} \in \conj(G/N))$.
 Let us note that for an Arens regular Banach algebra $\cA$, every quotient algebra $\cA/\cI$ where $\cI$ is a closed ideal of $\cA$ is Arens regular as well (see \cite[Corollary 3.15]{da2}). Therefore, if $\ell^1(\conj(G),\omega_\sigma)$ is Arens regular, for every normal subgroup $N$, $\ell^1(\conj(G/N),\tilde{\omega}_\sigma)$, which is isomorphic to $\ell^1(\conj(G),\omega_\sigma)/\ker(T_{\omega_\sigma})$, is Arens regular.
\end{rem}


In the final result of this section, we apply some techniques of \cite{yo} to show that for restricted direct product of hypergroups, product weights fail to admit  Arens regular algebras.

\begin{proposition}\label{p:non-Arens-of-special-products}
Let $\{ H_i\}_{i\in\Ind}$ be an infinite family of non-trivial discrete  hypergroups and for each $i\in\Ind$, $\omega_i$ is a weight on $H_i$ such that $\omega_i(e_{H_i})=1$ for all except finitely many $i\in\Ind$. Let $H=\bigoplus_{i\in\Ind}H_i$ and $\omega=\prod_{i\in\Ind}\omega_i$. Then $\ell^1(H,\omega)$ is not Arens regular.
\end{proposition}

\begin{proof}
Since $\Ind$ is infinite, suppose that $\Bbb{N}_0\times \Bbb{N}_0 \subseteq \Ind$. Define $v_n=(x_i)_{i\in \Ind}$  where $x_i=e_{H_i}$ for all $i\in \Ind\setminus (n,0)$  and $x_{(n,0)}$  be a non-identity element of $H_{(n,0)}$  for all $n\in \Bbb{N}$.   Similarly define $u_m=(x_i)_{i\in\Ind}$ where $x_i=e_{H_i}$ for all $i\in \Ind\setminus (0,m)$  and $x_{(0,m)}$  be a non-identity element of  $H_{(0,m)}$ for all $m\in\Bbb{N}$.
 Note that for each pair of elements $(n,m)\in\Bbb{N}\times \Bbb{N}$, $\supp(\delta_{v_n}*\delta_{u_m})$ forms a singleton in $H$; moreover, $
\omega(\delta_{v_n}*\delta_{u_m})=\omega(v_n)\omega(u_m)$. Hence, $(\delta_{v_n}*\delta_{u_m})_{(n,m)\in \Bbb{N}\times \Bbb{N}}$ forms a sequence of distinct elements in $\ell^1(H)$.

Let us define $f_n=\delta_{v_n}$ and $g_m=\delta_{u_m}$ for all $n,m\in\Bbb{N}$.
Suppose that  $A:=\{(v_n,u_m):\ n>m\}$ and  $\phi\in\ell^\infty(H)$ is the characteristic function of the subset $A$.
Clearly, $\kappa^{-1}(f_n)=\omega^{-1} f_n$ and $\kappa^{-1}(g_m)=\omega^{-1} g_m$ belong to $\ell^1(H,\omega)$ for all $n,m$ and $\kappa^*(\phi)=\omega \phi \in\ell^\infty(H,\omega^{-1})$, for    the Banach space isomorphism $\kappa: \ell^1(H,\omega)\rightarrow \ell^1(H)$ where $\kappa(f)=f\omega$ for each $f\in \ell^1(H,\omega)$. Note that
\begin{eqnarray*}
\langle \omega^{-1} f_n * \omega^{-1} g_m, \kappa^*(\phi)\rangle &=& \langle \omega^{-1} f_n * \omega^{-1} g_m, \omega \phi\rangle\\
 &=& \sum_{t\in H} (\omega^{-1} f_n * \omega^{-1} g_m)(t) \omega(t) \phi(t) \\
&=&   \frac{\omega(v_n*u_m)}{\omega(v_n)\omega(u_m)} \phi(\delta_{v_n}*\delta_{u_m}) \\
&=& \phi(\delta_{v_n}*\delta_{u_m}) = \left\{
\begin{array}{c c}
1 & \text{if $n > m$}\\
0 & \text{if $n\leq m$}
\end{array} \right.
\end{eqnarray*}
Let us recall that for each $n$ and $m$, $\norm{f_n}_{\ell^1(H,\omega)}=1$ and $\norm{g_m}_{\ell^1(H,\omega)}=1$. So $(f_n)_{n\in\Bbb{N}}$  and $(g_m)_{m\in\Bbb{N}}$, as two nets in the unit ball of $\ell^1(H,\omega)^{**}$, have two   subnets $(f_\alpha)_\alpha$ and $(g_\beta)_\beta$   such that $f_\alpha$  and $g_\beta$  converge  weakly$^*$ to some $F$  and $G$ in $\ell^1(H,\omega)^{**}$, respectively.
 Note that for the specific element $\phi$, defined above, $\langle F\Box G,\phi\rangle=0$
 while
$ \langle F\Diamond G,\phi\rangle=  1$.
Hence $F\Box G\neq F\Diamond G$ and consequently  $\ell^1(H,\omega)$ is not Arens regular.
\end{proof}


\end{section}



\begin{section}{Isomorphism to operator algebras}\label{s:Operator-algebra-weighted-hypergroups}

Let $(H,\omega)$ be a weighted discrete hypergroup. In this section, we study the existence of an algebra isomorphism from  $\ell^1(H,\omega)$ onto an operator algebra. A Banach algebra $\cA$ is called an operator algebra if there is a Hilbert space ${\mathcal H}$ such that $\cA$ is  a closed subalgebra of $\cB(\mathcal H)$.
Let $\mathcal A$ be a Banach algebra and $\m:{\mathcal A}\times {\mathcal A} \rightarrow {\mathcal A}$ is the bilinear (multiplication) mapping $\m(f,g)=fg$.
Then $\mathcal A$ is called  \ma{injective}, if $\m$ has a bounded extension from ${\mathcal A} \otimes_\epsilon {\mathcal A}$ into $\mathcal A$, where $\otimes_\epsilon$ is the injective tensor product. In this case, we denote the norm of $\m$ by $\norm{\m}_\epsilon$.
 \cite[Corollary 2.2.]{sa3}  proves that if a Banach algebra $\cA$ is injective then it is isomorphic to an operator algebra. But the converse also holds for weighted hypergroup algebras. The proof is similar to the group case in \cite[Theorem 2.8]{sa3} and it follows from the little Grothendieck inequality (see \cite{pisier}).
Note that a Banach algebra which is isomorphic to an operator algebra is always Arens regular (\cite[Corollary~2.5.4]{blecher}).
\vskip1.0em

Injectivity  of weighted group algebras has been studied before. Initially  Varopoulos, in \cite{va}, studied the group $\Bbb{Z}$ equipped with the weight $\sigma_\alpha(n)= (1+|n|)^\alpha$ for all $\alpha\geq 0$. This study  looked at { injectivity} of $\ell^1(\Bbb{Z},\sigma_\alpha)$. He showed that $\ell^1(\Bbb{Z},\sigma_\alpha)$ is injective if and only if $\alpha>1/2$.
 The manuscript~\cite{sa3}, which studied  the injectivity  question for a wider family of weighted group algebras, developed a machinery applying Littlewood multipliers. In particular, it partially extended Varopoulos's result to finitely generated groups with polynomial growth.
 Following the structure of \cite{sa3}, in this section,  we study the injectivity or equivalently isomorphism to operator algebras for  weighted hypergroup algebras.

In this section, $\cal A\otimes_\gamma B$ and  $\cal A\otimes_\epsilon B$ denote respectively the projective and injective tensor products of Banach spaces $\cal A$ and $\cal B$.

We know that $\ell^1(H,\omega)\otimes_\gamma \ell^1(H,\omega)$ is isometrically isomorphic to $\ell^1(H\times H,\omega \times \omega)$. Moreover, $\ell^1(H\times H,\omega \times \omega)^*$ is   $\ell^\infty(H\times H,\omega^{-1}\times \omega^{-1})$. Since  the injective tensor norm is minimal among all cross-norm Banach space tensor norms,   the identity map $\iota:\ell^1(H)\times \ell^1(H) \rightarrow \ell^1(H)\times \ell^1(H)$ may extend to a contractive mapping
\[
\iota: \ell^1(H) \otimes_\gamma \ell^1(H) \rightarrow \ell^1(H) \otimes_\epsilon \ell^1(H).
\]
Since, $\iota$ has a dense range,
 \begin{equation}\label{eq:e-sits-in-gamma}
 \iota^*: (\ell^1(H) \otimes_\epsilon \ell^1(H))^* \rightarrow (\ell^1(H) \otimes_\gamma \ell^1(H))^*=\ell^\infty(H\times H)
 \end{equation}
 is an injective mapping. Therefore, applying $\iota^*$, one may embed $(\ell^1(H)\otimes_\epsilon \ell^1(H))^*$ into $\ell^\infty(H\times H)$, as a linear  subspace of $\ell^\infty(H\times H)$.

Let $H$ be a discrete hypergroup. We define \ma{Littlewood multipliers} of $H$ to be the set of all functions $f:H\times H \rightarrow \Bbb{C}$ such that there exist functions $f_1,f_2: H\times H \rightarrow \Bbb{C}$ where $
f(x,y)=f_1(x,y) + f_2(x,y)$ for $x,y\in G$
such that
\[
\sup_{y\in H}\sum_{x\in H} |f_1(x,y)|^2 <\infty\ \ \text{and} \ \ \sup_{x\in H}\sum_{y\in H} |f_2(x,y)|^2 <\infty.
\]
We denote the set of all Littlewood multipliers by $T^2(H)$ and define the norm $\norm{\cdot}_{T^2(H)}$ by
\[
\norm{f}_{T^2(H)}:=\inf\left\{ \sup_{y\in H}\left( \sum_{x\in H} |f_1(x,y)|^2\right)^{1/2} + \sup_{x\in H}\left( \sum_{y\in H} |f_2(x,y)|^2\right)^{1/2} \right\}
\]
where the infimum is  taken over all possible decompositions $f_1,f_2$.
Note that for  a decomposition $f_1,f_2$ of    $f\in T^2(H)$,
\begin{eqnarray*}
\norm{f}_{\ell^\infty(H\times H)} = \sup_{x,y\in H} |f(x,y)| &\leq& \sup_{x,y\in H} |f_1(x,y)| + \sup_{x,y\in H} |f_2(x,y)|\\
&\leq& \sup_{y\in H}\left( \sum_{x\in H} |f_1(x,y)|^2\right)^{1/2} + \sup_{x\in H}\left( \sum_{y\in H} |f_2(x,y)|^2\right)^{1/2} <\infty,
\end{eqnarray*}
since for discrete space $H$, $\ell^2(H)\subseteq \ell^\infty(H)$ and $\norm{\cdot}_\infty \leq \norm{\cdot}_2$. Since $f_1,f_2$, in the previous equation are arbitrary, $\norm{f}_{\ell^\infty(H\times H)}\leq \norm{f}_{T^2(H)}$.
Hence $T^2(H)\subseteq \ell^\infty(H\times H)$. Furthermore, for each $\phi\in \ell^\infty(H\times H)$ and $f\in T^2(H)$, $f\phi\in T^2(H)$ and
$\norm{f\phi}_{T^2(H)} \leq \norm{f}_{T^2(H)}\norm{\phi}_\infty$.

\vskip1.5em

The following theorem is the hypergroup version of \cite[Theorem~2.7]{sa3}. Since the proof is very similar to  the group case, we  omit it here (although with all the details it can be found in \cite{ma-the}).  Here we use  $\Gr$ to denote \ma{Grothendieck's constant}.
First in his celebrated ``R{\'e}sum{\'e}",  Grothendieck   proved the existence of the constant $\Gr$ in  Grothendieck's inequality.  For a detailed account of {Grothendieck's constant}, its   history, and approximations look at \cite[Sections~3 and 4]{pisier}.

\begin{theorem}\label{t:littlewood}
Let  $I:T^2(H) \rightarrow (\ell^1 (H) \otimes_\gamma \ell^1(H))^*=\ell^\infty(H\times H)$ be the mapping which takes every element of $T^2(H)$ to itself as a bounded function on $H\times H$. Then $I(T^2(H))\subseteq \iota^*((\ell^1(H)\otimes _\epsilon \ell^1(H))^*)$ for the mapping $\iota^*$ defined in (\ref{eq:e-sits-in-gamma}).
Moreover,  $J:={\iota^*}^{-1}\circ I: T^2(H) \rightarrow (\ell^1(H)\otimes _\epsilon \ell^1(H))^*$ is bounded and   $\norm{J}\leq \Gr$.
\end{theorem}

From now on, we identify $(\ell^1(H)\otimes_\epsilon \ell^1(H))^*)$ with its image through the mapping $\iota^*$; hence, $J$ is the identity mapping which takes $T^2(H)$ into  $(\ell^1(H)\otimes_\epsilon \ell^1(H))^*$. We present our first main result of this section. This is  a generalization of \cite[Theorem~3.1]{sa3}.


\vskip2.0em


\begin{theorem}\label{t:injectiv-l^1(H,omega)-for-T2-weights}
Let $H$ be a discrete hypergroup and $\omega$ is a weight on $H$ such that $\Omega$, defined in (\ref{eq:Omega}), belongs to $T^2(H)$. Then $\ell^1(H,\omega)$ is injective and equivalently isomorphic to an operator algebra. Moreover,
for the multiplication map $\m$ on $\ell^1(H, \omega)\otimes_\epsilon \ell^1(H, \omega)$,  as defined before, $\norm{\m}_\epsilon \leq \Gr  \norm{\Omega}_{T^2(H)}$.
\end{theorem}

\begin{proof}
Let
$
\Gamma_\omega:\ell^1(H\times H, \omega \times \omega) \rightarrow \ell^1(H,\omega)$
such that
$\Gamma_\omega(f\otimes g):=f*g$
for  $f,g\in \ell^1(H,\omega)$. The adjoint of $\Gamma_\omega$, ${\Gamma}^*_\omega$, can be characterized   as follows.
\begin{eqnarray*}
{\Gamma}^*_\omega(\phi)(x,y) = \langle {\Gamma}^*_\omega(\phi),\delta_{x}\otimes\delta_{y}\rangle
= \langle \phi, \Gamma_\omega(\delta_x\otimes \delta_y)\rangle
= \langle \phi, \delta_x*\delta_y\rangle
\end{eqnarray*}
for all $\phi \in \ell^\infty(H,\omega^{-1})$ and $x,y \in H$.
  Now we define $L$ from $\ell^\infty(H)$ to $\ell^\infty(H\times H)$ such that the following diagram commutes,
\[
\xymatrix{
{\ell^\infty(H,\omega^{-1})} \ar[rr]^{{\Gamma}^*_{\omega}}  & & {\ell^\infty(H\times H, \omega^{-1} \times \omega^{-1})}\ar[d]_{R}\\
{\ell^\infty(H)}\ar[u]^{P}\ar[rr]^{L} & & {\ell^\infty(H\times H)}
 }
\]
where $P(\varphi)(x)=\varphi(x)\omega(x)$ for $\varphi\in \ell^\infty(H)$ and $R(\phi)(x,y)=\phi(x,y)\omega^{-1}(x)\omega^{-1}(y)$ for $\phi\in \ell^\infty(H\times H, \omega^{-1}\times\omega^{-1})$ and  $x,y\in H$.
Hence, one gets
\begin{eqnarray*}
L(\varphi)(x,y) = R\left( {\Gamma}^*_{\omega} \circ P(\varphi)\right)(x,y)
&=& \frac{\left({\Gamma}^*_{\omega} \circ P (\varphi)\right)(x,y)}{\omega(x) \omega(y)}\\
&=& \frac{{\Gamma}^*_{\omega} \left( \omega \varphi \right)(x,y)}{\omega(x) \omega(y)}\\
&=& \frac{ \langle \varphi \omega, \delta_x*\delta_y\rangle}{\omega(x) \omega(y)}\\
&=& \sum_{t\in H} \delta_x*\delta_y(t) \frac{\omega(t)}{\omega(x)\omega(y)} \varphi(t).
\end{eqnarray*}
for all $\varphi\in \ell^\infty(H)$.
Hence,
\[
\left| \sum_{t\in H} \delta_x*\delta_y(t) \frac{\omega(t)}{\omega(x)\omega(y)} \varphi(t)\right|
\leq  \sum_{t\in H} \delta_x*\delta_y(t) \frac{\omega(t)}{\omega(x)\omega(y)} |\varphi(t)|
 \leq \norm{\varphi}_\infty \Omega(x,y)
\]
So there is a function $v_\varphi:H\times H\rightarrow \Bbb{C}$ such that
\[
\frac{\langle\delta_x*\delta_y,\omega \varphi \rangle}{\omega(x) \omega(y)} = v_\varphi(x,y) \norm{\varphi}_\infty \Omega(x,y)
\]
and $\norm{v_\varphi}_\infty\leq 1$.
Therefore
$L(\varphi) =  \Lambda(\varphi) \Omega$
where $\Lambda(\varphi)(x,y):=v_\varphi(x,y) \norm{\phi}_\infty$
 for all $\varphi\in\ell^\infty(H)$.
Since $\Omega$ belongs to $T^2(H)$ and $T^2(H)$ is an $\ell^\infty(H\times H)$-module, $L(\varphi) \in T^2(H)$ and $\norm{L(\varphi)}_{T^2(H)} \leq \norm{\varphi}_\infty \norm{\Omega}_{T^2(H)}$.
Therefore $L(\ell^\infty(H))\subseteq T^2(H) \subseteq (\ell^1(H)\otimes_\epsilon \ell^1(H))^*$.

In this case, using the following diagram with $\cA=R^{-1}((\ell^1(H)\otimes_\epsilon \ell^1(H))^*)$,
\[
\xymatrix{
{\ell^\infty(H,\omega^{-1})} \ar[rr]^{{\Gamma}_\omega^*}  &  &
\cA \ar[rr]^{\iota} \ar[d]^{R|_{r}}  & & {\ell^\infty(H\times H, \omega^{-1} \times \omega^{-1})}\ar[d]_{R}\\
{\ell^\infty(H)}\ar[u]^{P}\ar[rr]^{L} & &  (\ell^1(H)\otimes_\epsilon \ell^1(H))^*  \ar[rr]^{\iota} &  & {\ell^\infty(H\times H)}
 }
\]

One can easily verify that $\cA=(\ell^1(H,\omega)\otimes_\epsilon \ell^1(H,\omega))^*$.
So, we  have shown that ${\Gamma}^*$ is a map projecting $\ell^\infty(H)$ into $(\ell^1(H)\otimes_\epsilon \ell^1(H))^*$ as a subset of $\ell^\infty(H\times H)$.
we see that ${\Gamma}_\omega^*$ is a map projecting $\ell^\infty(H, \omega^{-1})$ into $(\ell^1(H,\omega)\otimes_\epsilon \ell^1(H, \omega))^*$.
Hence, ${\Gamma}_\omega^*=\m^*$,  where $
\m$ is the multiplication extended to $\ell^1(H,\omega)\otimes_\epsilon \ell^1(H,\omega)$.
Therefore $\m$ is bounded and  $\norm{\m}=\norm{ \Gamma_\omega}=\norm{ R\Gamma_\omega P}=\norm{L}$. Moreover,
\begin{eqnarray*}
\norm{L(\varphi)}_{(\ell^1(H)\otimes^\epsilon\ell^1(H))^*} &\leq& \norm{J} \; \norm{\Gamma^*(\varphi)}_{T^2(H)}
\leq \Gr \left\|\Omega\right\|_{T^2(H)} \norm{\Lambda(\varphi)}_{\ell^\infty(H\times H)}\\
&\leq&  \Gr \left\|\Omega\right\|_{T^2(H)}  \norm{\varphi}_{\ell^\infty(H)}
\end{eqnarray*}
for all $\varphi\in \ell^\infty(H)$. Consequently, $\norm{\m}_\epsilon \leq \Gr \norm{\Omega}_{T^2(H)}$.
\end{proof}

\begin{eg}\label{eg:OP-of-SU(n)}
Let $\omega_\beta$ be the dimension weight defined on $\wSU(n)$ in Lemma~\ref{l:dimension-as-a-weight}. As we have shown in the proof of Proposition~\ref{p:Arens-of-l1(SU(n))}, for the polynomial weight $\omega_\beta$, $\beta\geq 0$, and $\mu,\nu \in \wSU(n)$,
\[
\Omega_\beta(\mu,\nu) \leq  C_n^\beta  (\frac{1}{1+\mu_1}+\frac{1}{1+\nu_1})^\beta
 \leq A_\beta C_n^\beta  \left( \frac{1}{(1+\mu_1)^\beta}+\frac{1}{(1+\nu_1)^\beta}\right),
\]
where $A_\beta=\min\{1,2^{\beta-1}\}$.
To study  $\norm{\cdot}_{T^2(\wSU(2))}$ for $\Omega_\beta$, let us note that for each $k\in \Bbb{N}\cup\{0\}$, there are less than $(1+k)^{n-2}$ many $\lambda=(\lambda_1,\ldots,\lambda_n) \in \wSU(n)$ such that $\lambda_1=k$.
Therefore
\[
\sum_{\lambda\in\wSU(n)} \frac{1}{(1+\lambda_1)^{2\beta}} \leq \sum_{k=0}^\infty \frac{(1+k)^{n-2}}{(1+k)^{2\beta}}
\]
where the right-hand side series converges  if and only if $2\beta - n +2>1$.
Therefore, for $\beta> (n-1)/2$, $\Omega_\beta\in T^2(\wSU(n))$ and by Theorem~\ref{t:injectiv-l^1(H,omega)-for-T2-weights}, $\ell^1(\wSU(2),\omega_\beta)$ is injective and consequently isomorphic to an operator algebra. Moreover,
note that
\begin{eqnarray*}
\norm{\Omega_\beta}_{T^2(\wSU(n))} &\leq& \left\|(\mu,\nu)\mapsto  \frac{A_\beta C_n^\beta}{1+\mu_1}+\frac{A_\beta C_n^\beta}{1+\nu_1}\right\|_{T^2(H)}\\
&\leq&  \sup_{\nu\in \wSU(n)}\left( \sum_{\mu\in \wSU(n)} \left|\frac{  A_\beta C_n^\beta }{1+\mu_1}\right|^2\right)^{1/2} \\
&+ &\sup_{\mu\in \wSU(n)}\left( \sum_{\nu\in \wSU(n)} \left|\frac{A_\beta C_n^\beta}{1+\nu_1}\right|^2\right)^{1/2}\\
&\leq& A_\beta C_n^\beta 2 \left( \sum_{k=0}^\infty \frac{1}{(1+k)^{2\beta-n+2}}\right)^{1/2}.
\end{eqnarray*}
Hence, for $A_\beta=\min\{1,2^{\beta-1}\}$,
\[
\norm{\m}_\epsilon \leq 2 \Gr A_\beta  C_n^\beta \left( \sum_{k=0}^\infty \frac{1}{(1+k)^{2\beta-n+2}}\right)^{1/2}
\]
\end{eg}

\vskip2.0em

\begin{cor}\label{c:injectiv-l^1(H,omega)-for-general-case}
Let $H$ be a discrete hypergroup and $\omega$ is a weakly  additive weight on $H$ with a corresponding constant $C>0$.
 Then $\ell^1(H,\omega)$ is injective if $ \sum_{x\in H}{\omega(x)^{-2}}<\infty$.
  Moreover,
\[
\norm{\m}_\epsilon \leq 2C\Gr  \left( \sum_{x\in H}\frac{1}{\omega(x)^2} \right)^{1/2}.
\]
\end{cor}

\begin{proof}
 Suppose that $ \sum_{x\in H} \omega(x)^{-2}<\infty$. Note that for each $t\in \supp(\delta_x*\delta_y)$,
\begin{eqnarray*}
\frac{\omega(t)}{\omega(x)\omega(y)} \leq C\frac{\omega(x)+\omega(y)}{\omega(x)\omega(y)} = \frac{C}{\omega(x)}+\frac{C}{\omega(y)}.
\end{eqnarray*}
 Thus, for the functions $f_1(x,y)={\omega(x)}^{-1}$ and $f_2(x,y)={\omega(y)}^{-1}$,
\begin{eqnarray*}
\norm{\Omega}_{T^2(H)} &\leq& \left\|(x,y)\mapsto \frac{C}{\omega(x)}+\frac{C}{\omega(y)}\right\|_{T^2(H)}\\ &\leq&   \left(\sup_{y\in H}\left( \sum_{x\in H} \left|\frac{C}{\omega(x)}\right|^2\right)^{1/2} + \sup_{x\in H}\left( \sum_{y\in H} \left|\frac{C}{\omega(y)}\right|^2\right)^{1/2}\right) \leq  2C \left(\sum_{x\in H} \frac{1}{\omega(x)^2}\right)^{\frac{1}{2}}.
\end{eqnarray*}
Consequently, by Theorem~\ref{t:injectiv-l^1(H,omega)-for-T2-weights}, $\ell^1(H,\omega)$ is injective and $\norm{\m}_\epsilon$ satisfies the mentioned inequality.
\end{proof}

\begin{eg}\label{eg:chebyshev-OP}
Let $\omega_f$ be the weight constructed by the group weight  admitted by a positive increasing function $f$ (see Example~\ref{eg:Chebyshev-weight-Arens}). One can see that, if  \[
\sum_{n\in \Nat_0} \frac{1}{f(n)^2}<\infty\ \ \text{and}\ \ \sup_{n,m\in \Nat_0} \frac{ f(n+m)}{f(n)+ f(m)} <\infty,
\]
then $\omega_f$ satisfies the conditions of Corollary~\ref{c:injectiv-l^1(H,omega)-for-general-case} and therefore, $\ell^1(\Nat_0,\omega_f)$ is isomorphic to an operator algebra. On the other hand, $\ell^1(\Nat_0,\omega_f)$ can be embedded (isomorphically as a Banach algebra) into $A(\Bbb{T},\sigma_f)$ which is not isomorphic to any operator algebra (as it is not even Arens regular, see Example~\ref{eg:Chebyshev-Arens}).
\end{eg}

\begin{rem}
Note that the assumed  condition for $f$ in Example~\ref{eg:chebyshev-OP} implies the  Arens regularity condition required in Example~\ref{eg:Chebyshev-Arens}. Compare it with this know fact that every Banach algebra which is isomorphic to an operator algebra is Arens regular.
\end{rem}


\begin{rem}\label{r:operator-algebra-of-conj(SL(2,2^n)}
Let  $(\conj(G), \omega_\alpha)$  be the weighted hypergroup defined in Example~\ref{eg:SL(2,2n)-Arens-regularity}.  Note that  $\omega_\alpha$ is a weakly additive weight.
One can straightforwardly  show that $\sum_{C\in \conj(G)} \omega(C)^{-2} =\infty$.
Hence, not all weakly additive weights are satisfying the other condition mentioned in Corollary~\ref{c:injectiv-l^1(H,omega)-for-general-case}.
\end{rem}


For finitely generated hypergroups, we showed that  that polynomial weights  are  weakly additive.  In the following, we study operator algebra isomorphism for weighted hypergroup algebras with polynomial weights. Developing a machinery which relates exponential weights to polynomial ones, we also study exponential weights in Subsection~\ref{ss:OP-exponential-weights}. For the case that $H$ is a group, this has been achieved in \cite{sa3}

\begin{cor}\label{c:injectiv-l^1(H,omega)-for-polynomial-weights}
Let $H$ be a finitely generated hypergroup. If $F$ is a generator of $H$ such that $|F^{*n}|\leq D n^d$ for some $d,D>0$ and $\omega_\beta$ is the polynomial weight on $H$ associated to $F$. Then $\ell^1(H,\omega_\beta)$ is injective if $2\beta>d+1$. Moreover, for $C=\min\{1,2^{\beta-1}\}$,
\[
\norm{\m}_\epsilon \leq  2C\Gr \left( 1 + \sum_{n=1}^\infty \frac{ D n^d}{(1+n)^{2\beta}} \right)^{1/2} .
\]
\end{cor}

\begin{proof}
To prove this corollary, we mainly rely on Corollary~\ref{c:injectiv-l^1(H,omega)-for-general-case}. Recall that $\omega_\beta$ is weakly additive whose constant is $C=\min\{1,2^{\beta-1}\}$.
To show the desired bound for $\norm{\m}_\epsilon$, note that
\begin{eqnarray*}
\sum_{x\in H} \frac{1}{\omega_\beta(x)^2} &=& \sum_{x\in H} \frac{1}{(1+\tau(x))^{2\beta}}
= \sum_{n=0}^\infty \sum_{\{x\in F^n \setminus F^{n-1}\}} \frac{1}{(1+n)^{2\beta}}\\
&\leq& 1 + \sum_{n=1}^\infty \frac{|F^n|}{(1+n)^{2\beta}}
\leq  1 + \sum_{n=1}^\infty \frac{ D n^d}{(1+n)^{2\beta}}
\end{eqnarray*}
which is convergent if $2\beta>d+1$.
\end{proof}




\begin{eg}\label{eg:operator-algebra-of-polynomial}
For a polynomial hypergroup $\Bbb{N}_0$, as a finitely generated hypergroup with the generator $F=\{0,1\}$, we have $|F^{*n}|=n+1 \leq 2n$, as we have seen before. 
By Corollary~\ref{c:injectiv-l^1(H,omega)-for-polynomial-weights}, for the  polynomial weight $\omega_\beta$ with $\beta>1$ associated to $F$, $\ell^1( \Bbb{N}_0,\omega_\beta)$ is injective.
For $C=\min\{1,2^{\beta-1}\}$, Corollary~\ref{c:injectiv-l^1(H,omega)-for-general-case} implies that
\[
\norm{\m}_\epsilon \leq 2C\Gr \left( \sum_{n=1}^\infty \frac{ 1}{n^{2\beta}} \right)^{1/2}.
\]
\end{eg}



\begin{subsection}{Hypergroups with exponential weights}\label{ss:OP-exponential-weights}

The other class of weights introduced for finitely generated hypergroups  is the class of exponential weights. As we mentioned before, unlike polynomial weights, exponential weights are not necessarily weakly additive. In this subsection, following \cite{sa3}, we  study operator algebra isomorphism of these weights by studying the cases for them $\Omega$ belongs to $T^2(H)$.  The following lemma is   a hypergroup adaptation of \cite[Theorem~3.3]{sa3}. Since the proof is similar to the one of \cite[Lemma~B.2]{ghan}, we omit it here.

\begin{lemma}\label{l:weight-coming-from-p&q}
Suppose that $0<\alpha<1$, $C>0$, and $\beta\geq \max\left\{ 1, \frac{6}{C\alpha(1-\alpha)}\right\}$.
Define the functions $p:[0,\infty) \rightarrow \Bbb{R}$ and $q:(0,\infty)\rightarrow \Bbb{R}$ by
$p(x):=Cx^{\alpha}-\beta \ln(1+x)$ and $q(x):=\frac{p(x)}{x}$. 
Let $H$ be a finitely generated hypergroup with a symmetric generator $F$   and  $\omega: H\rightarrow (0,\infty)$ such that
\[
\omega(x)=e^{p(\tau_F(x))}=e^{\tau_F(x) q(\tau_F(x))}\ \ \text{for all $x\in H$}.
\]
Then $\omega(t)\leq M \omega(x) \omega(y)$ for all $t,x,y\in H$ such that $t\in x*y$ where
\[
M=\max\{e^{p(z_1)-p(z_2)-p(z_3)}: z_1,\ z_2,z_3\in [0,2K]\cap \Bbb{N}_0\}
\]
and
\[
K=\left( \frac{\beta^2}{C\alpha(1-\alpha)}\right)^{1/\alpha}.
\]
\end{lemma}

\begin{theorem}\label{t:injectiv-l^1(H,sigma)-for-exponential-weights}
Let $H$ be a finitely generated hypergroup. If $F$ is a symmetric generator of $H$ such that $|F^{*n}|\leq D n^d$ for some $d,D>0$ and $\sigma_{\alpha,C}$ is an exponential weight on $H$ for some $0<\alpha<1$ and $C>0$. Then $\ell^1(H,\sigma_{\alpha,C})$ is injective and equivalently isomorphic to an operator algebra.
\end{theorem}

\begin{proof}
Let $\omega_\beta$ be the weight defined in Lemma~\ref{l:weight-coming-from-p&q}.
We define a function $\omega:H \rightarrow (0,\infty)$ by
\[
\omega(x):=\frac{\sigma_{\alpha,C}(x)}{\omega_\beta(x)} = e^{C\tF(x)^\alpha - \beta \ln(1+\tF(x))}\ \ (x\in H)
\]
where $\omega_\beta$ is the polynomial weight defined on $H$ associated to $F$ and
\[
\beta>\max \{ 1, \frac{6}{C\alpha(1-\alpha)}, \frac{d+1}{2} \}.
\]
 Therefore, by Lemma~\ref{l:weight-coming-from-p&q}, $\omega(t)\leq M \omega(x)\omega(y)$ for some $M>0$ and all $t,x,y\in H$ such that $t\in x*y$.
Therefore
\[
\frac{\sigma_{\alpha,C}(t)}{\sigma_{\alpha,C}(x)\sigma_{\alpha,C}(y)} \leq M \frac{\omega_\beta(t)}{\omega_\beta(x)\omega_\beta(y)}.
\]
Therefore,
\[
\frac{\sigma_{\alpha,C}(t)}{\sigma_{\alpha,C}(x)\sigma_{\alpha,C}(y)} \leq M' \left( \frac{1}{(1+\tau(x))^\beta}+\frac{1}{(1+\tau(y))^\beta}\right)
\]
for a modified constant $M'>0$.
Therefore by the proof of Corollary~\ref{c:injectiv-l^1(H,omega)-for-polynomial-weights}, $\Omega_{\sigma_{\alpha,C}}\in T^2(H)$.  Now  Theorem~\ref{t:injectiv-l^1(H,omega)-for-T2-weights} finishes the proof.
\end{proof}

\begin{eg}


As a result of Theorem~\ref{t:injectiv-l^1(H,sigma)-for-exponential-weights}, and to follow  Example~\ref{eg:operator-algebra-of-polynomial}, if $H$ is a polynomial hypergroup on  $\Bbb{N}_0$, for each exponential weight $\sigma_{\alpha,C}$ for $0<\alpha<1$ and $C>0$, $\ell^1(H,\sigma_{\alpha,C})$ is injective. Note that  this class of hypergroups includes $\wSU(2)$.
\end{eg}

\end{subsection}

\end{section}

\section*{Acknowledgements}
For this research, the first author was supported by a Ph.D. Dean's Scholarship at University of Saskatchewan and a Postdoctoral Fellowship form the Fields Institute For Research In Mathematical Sciences and University of Waterloo. The second name author was also supported by NSERC Discovery grant no 409364 and a generous support from the Fields Institute.  The first named author also would like to express his deep gratitude to Yemon Choi and Nico Spronk  for several constructive discussions and suggestions which improved the paper significantly. The authors also would like to thank the referee for his many productive comments.

\appendix
\begin{section}{Appendix: Lifting weights from $\Bbb{Z}$ to weights on $A({\SU}(2))$}\label{a:SU(2)}

In this appendix, we briefly present a method to construct non-central   weights on $A({\SU}(2))$ which are related to Example~\ref{eg:weight-driven-from-Z-weights}. Here  $\lambda_G$ denotes the left regular representation of a compact group $G$ on $L^2(G)$ and $VN(G)$ denotes the group von Neumann algebra generated by $\lambda_G$.
We also identify $\bT$ with the (closed) subgroup of all matrices
\[
\left[
\begin{array}{c c}
{t} & 0 \\
0 & \overline{t}
\end{array}
\right], \ \ \ (t\in \bT)
\]
in $\SU(2)$. It is an immediate consequence of Herz's restriction theorem that there is a canonical embedding of $VN(\bT)$ into $VN(\SU(2))$. More precisely, the mapping $\Gamma: VN(\bT) \to VN(\SU(2))$ defined by
$$\Gamma(\lambda_\bT(f))=\int_\bT f(t) \lambda_{\SU(2)}(f) dt$$
is a $\text{weak}^*$-$\text{weak}^*$ isometric $*$-algebra homomorphism. We note that the integration in the definition of $\Gamma$ is the Bochnor integration in 
the weak operator topology of $B(L^2(\SU(2)))$. 

Now suppose that $\sigma$ is a (group) weight on $\Bbb{Z}$ which is bounded below by some $\delta>0$, i.e. $\sigma^{-1}$ belongs to  $\ell^\infty(\Zat)$.  Through the Fourier transform $\mathcal{F}$ on $\Zat$, $\mathcal{F}(\sigma^{-1})$ is an element in $VN(\bT)$ defined by
$$\mathcal{F}(\sigma^{-1}) \chi_k = \sigma(k)^{-1} \chi_k,$$ where $\chi_k(t)=t^k$ ($k \in \Bbb{Z}$) are the characters of $\mathbb{T}$. 

We now consider the element $\Gamma(\mathcal{F}(\sigma^{-1}))$ in $VN(\SU(2))$. Since $\SU(2)$ is compact, we can write $\lambda_{\SU(2)}$ as the direct sum of the irreducible unitary representations of $\SU(2)$. Moreover, if we take $\widehat{\SU}(2)=\{ \pi_\ell: \ell \in \Nat_0\}$, where each $\pi_\ell$ is a representation of dimension $\ell+1$,
then, by a straightforward computation based on \cite[Theorem~29.18]{he2}, we have that $\Gamma(\mathcal{F}(\sigma^{-1}))(\pi_\ell)$ is the diagonal matrix $\operatorname{diag}( {\sigma(-\ell)}^{-1} ,  {\sigma(-\ell+2)}^{-1}, \ldots,  {\sigma(\ell-2)}^{-1},  {\sigma(\ell)}^{-1} )$.
Therefore, if we define, 
\begin{equation}\label{eq:W}
W  =  \bigoplus_{\ell} - \left[
\begin{array}{c c c c c}
 {\sigma(-\ell)} & 0 & \cdots & 0 & 0 \\
0 & {\sigma(-\ell+2)} & \cdots & 0 & 0\\
\vdots &  \vdots & \ddots & \vdots & \vdots\\
0 & 0 & \cdots & {\sigma(\ell-2)} & 0 \\
0 & 0 & \cdots & 0 &  {\sigma(\ell)}
\end{array}
\right],
\end{equation}
then $W$ is a (possibly unbounded) operator $L^2(\SU(2))$ with $W^{-1}=\Gamma(\mathcal{F}(\sigma^{-1}))\in VN(\SU(2))$. Moreover, it is straightforward to check that $W$ 
satisfies the assumptions in \cite[Definition~2.4]{LeSa} so that, in particular, it is a Fourier algebra weight on $A(\SU(2))$. We can apply the formula in Definition~\ref{d:defining-weights-on-^G} to $W$ and define the (hypergroup) weight $\omega_\sigma$ on $\SU(2)$ presented in Example~\ref{eg:weight-driven-from-Z-weights}. 

\end{section}



\footnotesize

\bibliographystyle{plain}
\bibliography{Bibliography}

\def\cprime{$'$} \def\cprime{$'$}
\begin{thebibliography}{10}

\bibitem{SL(2F)}
J.~Adams.
\newblock Character table of ${SL}(2,\mathbb{F})$.
\newblock {\em
  \href{http://www2.math.umd.edu/\~jda/characters/sl2/}{http://www2.math.umd.edu/\~jda/characters/sl2/}.}

\bibitem{ma-the}
Mahmood Alaghmandan.
\newblock {\em weighted hypergroups and some questions in abstract harmonic
  analysis}.
\newblock ProQuest LLC, Ann Arbor, MI, 2013.
\newblock Thesis (Ph.D.)--The University of Saskatchewan (Canada).

\bibitem{zag}
Mahmood Alaghmandan and Nico Spronk.
\newblock Amenability properties of the algebra of central fourier algebra of a
  compact group.
\newblock {\em (submited)}.

\bibitem{bh1}
H.~N. Bhattarai and J.~W. Fernandez.
\newblock Joins of double coset spaces.
\newblock {\em Pacific J. Math.}, 98(2):271--280, 1982.

\bibitem{blecher}
David~P. Blecher and Christian Le~Merdy.
\newblock {\em Operator algebras and their modules---an operator space
  approach}, volume~30 of {\em London Mathematical Society Monographs. New
  Series}.
\newblock The Clarendon Press, Oxford University Press, Oxford, 2004.
\newblock Oxford Science Publications.

\bibitem{bl}
Walter~R. Bloom and Herbert Heyer.
\newblock {\em Harmonic analysis of probability measures on hypergroups},
  volume~20 of {\em de Gruyter Studies in Mathematics}.
\newblock Walter de Gruyter \& Co., Berlin, 1995.

\bibitem{lee1}
Beno{\^{\i}}t Collins, Hun~Hee Lee, and Piotr {\'S}niady.
\newblock Dimensions of components of tensor products of representations of
  linear groups with applications to {B}eurling-{F}ourier algebras.
\newblock {\em Studia Math.}, 220(3):221--241, 2014.

\bibitem{yo}
I.~G. Craw and N.~J. Young.
\newblock Regularity of multiplication in weighted group and semigroup
  algebras.
\newblock {\em Quart. J. Math. Oxford Ser. (2)}, 25:351--358, 1974.

\bibitem{da2}
H.~G. Dales and A.~T.-M. Lau.
\newblock The second duals of {B}eurling algebras.
\newblock {\em Mem. Amer. Math. Soc.}, 177(836):vi+191, 2005.

\bibitem{ful}
William Fulton and Joe Harris.
\newblock {\em Representation theory}, volume 129 of {\em Graduate Texts in
  Mathematics}.
\newblock Springer-Verlag, New York, 1991.
\newblock A first course, Readings in Mathematics.

\bibitem{gh85}
F.~Ghahramani and A.~R. Medgalchi.
\newblock Compact multipliers on weighted hypergroup algebras.
\newblock {\em Math. Proc. Cambridge Philos. Soc.}, 98(3):493--500, 1985.

\bibitem{gh86}
F.~Ghahramani and A.~R. Medgalchi.
\newblock Compact multipliers on weighted hypergroup algebras. {II}.
\newblock {\em Math. Proc. Cambridge Philos. Soc.}, 100(1):145--149, 1986.

\bibitem{ghan}
Mahya Ghandehari, Hun~Hee Lee, Ebrahim Samei, and Nico Spronk.
\newblock Some {B}eurling-{F}ourier algebras on compact groups are operator
  algebras.
\newblock {\em Trans. Amer. Math. Soc.}, 367(10):7029--7059, 2015.

\bibitem{he2}
Edwin Hewitt and Kenneth~A. Ross.
\newblock {\em Abstract harmonic analysis. {V}ol. {II}: {S}tructure and
  analysis for compact groups. {A}nalysis on locally compact {A}belian groups}.
\newblock Die Grundlehren der mathematischen Wissenschaften, Band 152.
  Springer-Verlag, New York, 1970.

\bibitem{je}
Robert~I. Jewett.
\newblock Spaces with an abstract convolution of measures.
\newblock {\em Advances in Math.}, 18(1):1--101, 1975.

\bibitem{kam}
Rajab~Ali Kamyabi-Gol.
\newblock {\em Topological center of dual branch algebras associated to
  hypergroups}.
\newblock ProQuest LLC, Ann Arbor, MI, 1997.
\newblock Thesis (Ph.D.)--University of Alberta (Canada).

\bibitem{la*}
Rupert Lasser.
\newblock Orthogonal polynomials and hypergroups.
\newblock {\em Rend. Mat. (7)}, 3(2):185--209, 1983.

\bibitem{LeSa}
Hun~Hee Lee and Ebrahim Samei.
\newblock Beurling-{F}ourier algebras, operator amenability and {A}rens
  regularity.
\newblock {\em J. Funct. Anal.}, 262(1):167--209, 2012.

\bibitem{sa3}
Hun~Hee Lee, Ebrahim Samei, and Nico Spronk.
\newblock Some weighted group algebras are operator algebras.
\newblock {\em Proc. Edinb. Math. Soc. (2)}, 58(2):499--519, 2015.

\bibitem{sp}
Jean Ludwig, Nico Spronk, and Lyudmila Turowska.
\newblock Beurling-{F}ourier algebras on compact groups: spectral theory.
\newblock {\em J. Funct. Anal.}, 262(2):463--499, 2012.

\bibitem{pisier}
Gilles Pisier.
\newblock Grothendieck's theorem, past and present.
\newblock {\em Bull. Amer. Math. Soc. (N.S.)}, 49(2):237--323, 2012.

\bibitem{rob}
Derek J.~S. Robinson.
\newblock {\em A course in the theory of groups}, volume~80 of {\em Graduate
  Texts in Mathematics}.
\newblock Springer-Verlag, New York, second edition, 1996.

\bibitem{ulger}
A.~{\"U}lger.
\newblock Arens regularity of weakly sequentially complete {B}anach algebras.
\newblock {\em Proc. Amer. Math. Soc.}, 127(11):3221--3227, 1999.

\bibitem{va}
Nicholas~Th. Varopoulos.
\newblock Sur les quotients des alg\`ebres uniformes.
\newblock {\em C. R. Acad. Sci. Paris S\'er. A-B}, 274:A1344--A1346, 1972.

\bibitem{voit*}
Michael Voit.
\newblock Laws of large numbers for polynomial hypergroups and some
  applications.
\newblock {\em J. Theoret. Probab.}, 3(2):245--266, 1990.

\end{thebibliography}
$\line(1,0){200}$

{\footnotesize
{\bf Mahmood Alaghmandan}\newline\indent
Department of Mathematical Sciences, Chalmers University of Technology and  University of Gothenburg, Gothenburg SE-412 96, Sweden

\texttt{mahala@chalmers.se}
\vskip1.0em
{\bf Ebrahim Samei}\newline\indent
Department of Mathematics and Statistics,
University of Saskatchewan,
142 Wiggins road,
Saskatoon, SK S7N 5E6, Canada\newline\indent
\texttt{samei@math.usask.ca}
}

\end{document}